\documentclass[11pt,reqno]{amsart}
\usepackage[tmargin=1in,bmargin=1in,rmargin=1in,lmargin=1in]{geometry}
\usepackage[breaklinks=true]{hyperref}

\usepackage{xcolor}

\usepackage{amsmath,amsthm,amsfonts,amssymb,amscd,enumerate}
\usepackage{mathrsfs,tocvsec2}
\usepackage{algorithm,algpseudocode}

\theoremstyle{plain}
\newtheorem{theorem}{Theorem}[section]
\newtheorem{lemma}[theorem]{Lemma}
\newtheorem{prop}[theorem]{Proposition}
\newtheorem{cor}[theorem]{Corollary}
\newtheorem{utheorem}{\textrm{\textbf{Theorem}}}

\theoremstyle{definition}
\newtheorem{defn}[theorem]{Definition}
\newtheorem{remark}[theorem]{Remark}
\newtheorem{example}[theorem]{Example}
\newtheorem{question}[theorem]{Question}

\numberwithin{equation}{section}
\numberwithin{algorithm}{section}

\DeclareMathOperator{\Id}{Id}
\DeclareMathOperator{\In}{In}
\newcommand{\R}{\mathbb{R}}
\newcommand{\C}{\mathbb{C}}
\newcommand{\D}{\mathbb{D}}
\newcommand{\F}{\mathbb{F}}
\newcommand{\LL}{\widetilde{\mathbf{L}}}
\newcommand{\bp}[1]{\ensuremath{\mathbb P}_\mu \left( #1 \right)}
\newcommand{\bbp}{\ensuremath{\mathbb P}}
\newcommand{\symm}{\mathbb{S}}
\newcommand{\beps}{\boldsymbol{\epsilon}}
\newcommand{\epsrev}{\overset{\leftarrow}{\epsilon}}
\newcommand{\rev}[1]{\overset{\leftarrow}{#1}}
\newcommand{\floo}[1]{\lfloor #1 \rfloor}
\newcommand{\ceil}[1]{\lceil #1 \rceil}
\newcommand{\half}[1]{#1_{\frac{1}{2}}}
\newcommand{\tangle}[1]{\langle #1 \rangle}
\renewcommand{\geq}{\geqslant}
\renewcommand{\leq}{\leqslant}

\begin{document}

\title[Cholesky-factoring symmetric matrices, Riemannian geometry, and
random matrices]{Cholesky decomposition for symmetric matrices,\\
Riemannian geometry, and random matrices}

\author{Apoorva Khare}
\address[A.~Khare]{Department of Mathematics, Indian Institute of
Science, Bangalore, India; and Analysis \& Probability Research Group,
Bangalore, India}
\email{\tt khare@iisc.ac.in}

\author{Prateek Kumar Vishwakarma}
\address[P.K.~Vishwakarma]{D\'epartement de math\'ematiques et de
statistique, Universit\'e Laval, Qu\'ebec, Canada}
\email{\tt prateek-kumar.vishwakarma.1@ulaval.ca,
prateekv@alum.iisc.ac.in}

\begin{abstract}
For each $n \geq 1$ and sign pattern $\epsilon \in \{ \pm 1 \}^n$, we
introduce a cone of real symmetric matrices $LPM_n(\epsilon)$: those with
leading principal $k \times k$ minors of signs $\epsilon_k$. These cones
are pairwise disjoint and their union $LPM_n$ is an open dense cone in
all symmetric matrices; they subsume positive and negative definite
matrices, and symmetric (P-,) N-, PN-, almost P-, and almost N- matrices.
We show that each $LPM_n$ matrix $A$ admits an uncountable family of
Cholesky-type factorizations -- yielding a unique lower triangular matrix
$L$ with positive diagonals -- with additional attractive properties:
\textit{(i)}~each such factorization is algorithmic; and
\textit{(ii)}~each such Cholesky map $A \mapsto L$ is a smooth
diffeomorphism from $LPM_n(\epsilon)$ onto an open Euclidean ball.

We then show that
\textit{(iii)}~the (diffeomorphic) balls $LPM_n(\epsilon)$ are isometric
Riemannian manifolds as well as isomorphic abelian Lie groups, each
equipped with a translation-invariant Riemannian metric (and hence
Riemannian means/barycentres).
Moreover,
\textit{(iv)}~this abelian metric group structure on each
$LPM_n(\epsilon)$ -- and hence the log-Cholesky metric on Cholesky space
-- yields an isometric isomorphism onto a finite-dimensional Euclidean
space. The complex version of this also holds.

In the latter part, we show that the abelian group $PD_n$ of positive
definite matrices, with its bi-invariant log-Cholesky metric, is
precisely the identity-component of a larger group with an alternate
metric: the open dense cone $LPM_n$. This also holds for Hermitian
matrices over several subfields $\mathbb{F} \subseteq \mathbb{C}$. As a
result,
\textit{(v)}~the groups $LPM_n^{\mathbb{F}}$ and
$LPM_\infty^{\mathbb{F}}$ admit a rich probability theory, and the cones
$LPM_n(\epsilon), TPM_n(\epsilon)$ admit Wishart densities with signed
Bartlett decompositions.
\end{abstract}

\subjclass[2020]{
15A23, 
15B48, 
53C22, 
46C05 (primary); 
22E99, 
47A64, 
60B10, 
60B20, 
60E15 (secondary)} 

\keywords{Cholesky decomposition,
Cholesky factorization,
sign pattern,
LPM matrix,
TPM matrix,
Riemannian metric,
log-Cholesky metric,
Lie group,
Hilbert space,
SSRPM matrix,
Hoffmann-J{\o}rgensen inequality, 
Ottaviani--Skorohod inequality, 
Mogul'skii inequality,
L\'evy--Ottaviani inequality,
L\'evy equivalence,
Lorentz--Gram matrix,
Wishart distribution,
signed Bartlett decomposition,
inverse Wishart density,
Cholesky-normal distribution,
lognormal distribution.
}

\date{\today}

\maketitle

\settocdepth{section}
\tableofcontents

\section{Introduction and main results}

One hundred years since Andr\'e-Louis Cholesky's fundamental
factorization of a positive definite matrix was (posthumously) published
in 1924~\cite{Benoit}, the Cholesky decomposition is ubiquitous in
mathematics, the broader sciences, and applied fields.

The Cholesky factorization has since been extended to generic symmetric
matrices, via the LDU decomposition. Namely, let $LPM_n$ denote the dense
subset of real symmetric $n \times n$ matrices, which have all nonzero
leading principal minors. Then e.g.\ \cite[Corollary~3.5.6]{HJ} says that
every matrix in $LPM_n$ admits a unique decomposition into $LDU$, where
$L$ is unit lower triangular, $D$ is diagonal (and depends on $A$), and
$U$ is unit upper triangular (i.e., $L^T,U$ are unipotent and $D$ is in
the torus inside the Borel subgroup of $GL_n(\R)$).

This extension from the positive definite cone $PD_n$ to $LPM_n$ has
already been extensively studied, and is of much use. Our original goal
in this work was to provide for all $LPM_n$ matrices $A$ a parallel
factorization to LDU, but in which $D$ depends on the leading principal
minors of $A$ only through their signs. This search has proved richly
rewarding: it led us to
(a)~uncover for every matrix in $LPM_n$, an uncountable family of such
``unique Cholesky-type factorizations'', which we provide below.
Each of these factorizations has further attractive features:
(b)~It is algorithmic in nature, akin to the ``usual'' Cholesky
decomposition; but it differs from it and from the LDU decomposition.
(c)~It involves smooth diffeomorphisms from certain matrix sub-cones
$LPM_n(\epsilon)$ (defined below) to the positive orthant, hence to an
open Euclidean ball.
Our results also (d)~simultaneously reveal a Riemannian manifold
structure and two abelian Lie group structures on these open sub-cones:
one individually for each sub-cone and the other for their union $LPM_n$.
(e)~These structures facilitate defining probability densities on each
sub-cone $LPM_n(\epsilon)$ -- and we also introduce a novel density on
the cone $PD_n$ itself -- which enable stochastic modeling and
statistical sampling from these cones.

\subsection{Two motivations}\label{Smotivations}

Before elaborating on these rich findings, we begin with two motivations
for studying these cones $LPM_n(\epsilon)$ and seeking a parallel
factorization to LDU with added properties. Our theoretical motivation
arises from to the parallel theory of total positivity. Recall that a
real matrix is totally positive if the determinant of every square
submatrix (termed a ``minor'') is positive. Such matrices have been
studied for over a century, in numerous subfields of mathematics:
analysis, approximation theory, matrix theory, particle systems,
probability, representation theory, combinatorics, integrable systems,
and Gabor analysis among others. See e.g.\ the monograph~\cite{Karlin}
for more on these matrices and more general kernels.

A 1950s result due to Whitney~\cite{Whitney} and Loewner~\cite{Loewner}
shows how to factorize totally positive matrices as products of
bi-diagonal matrices. This was taken forward by
Berenstein--Fomin--Zelevinsky~\cite{BFZ} (following
Lusztig~\cite{Lusztig}), who showed that this factorization yields a
\textit{diffeomorphism} onto an orthant. We present their result; in it
and beyond, $\Id_n$ denotes the $n \times n$ identity matrix, and
$E_{u,v}$ an elementary matrix with $1$ at $(u,v)$ position and zero
otherwise.

\begin{theorem}[\cite{BFZ}]\label{TBFZ}
Fix positive real tuples ${\bf w} = (w_{jk} : 0 < j \leq k < n)$ and
${\bf w'} = (w'_{jk} : 0 < j \leq k < n)$, and ${\bf d} = (d_1, \dots,
d_n)$. The map sending these $n^2$ positive scalars (in this order) to
\[
A({\bf w}, {\bf w'}, {\bf d}) :=
\prod_{j=1}^{n-1} \prod_{k=n-1}^j (\Id_n + w_{jk} E_{k+1,k}) \cdot
\prod_{j=n-1}^1 \prod_{k=j}^{n-1} (\Id_n + w'_{jk} E_{k,k+1}) \cdot
{\rm diag}(d_1, \dots, d_n)
\]
is a diffeomorphism from $(0,\infty)^{n^2}$ onto $TP_{n \times n}$ (the
$n \times n$ totally positive matrices).
\end{theorem}

This factorization-diffeomorphism from the 1990s has been involved in
important advances, including canonical bases, cluster algebras, plabic
graphs, and the study of the totally nonnegative Grassmannian. Note also that
if one composes with e.g.\ the map $\Psi : B_{\R^{n^2}}({\bf 0}_{n^2},1)
\mapsto (0,\infty)^{n^2}$, sending ${\bf 0}_{n^2}$ to ${\bf 1}_{n^2}$
(the all-ones vector) and every other point
\begin{equation}\label{Epsi}
x = (x_1, \dots, x_{n^2}) \ \mapsto \ \Psi(x) := \left( \exp ( \tan (
\frac{\pi \| x \|}{2} ) \frac{x_1}{\| x \|} ), \dots, \exp ( \tan (
\frac{\pi \| x \|}{2} ) \frac{x_{n^2}}{\| x \|} ) \right)
\end{equation}
then one obtains a smooth diffeomorphism (e.g.~\cite{Stack}) from $TP_{n
\times n}$ onto an open ball. (As a related fact in the theory of the
totally nonnegative Grassmannian: it too is homeomorphic to a ball, as
shown recently~\cite{GKL}  by Galashin--Karp--Lam.)

A more general variant of total positivity, morally going back to the
works of Descartes and Laguerre, involves the \textit{Strictly Sign
Regular (SSR)} matrices. These are real $m \times n$ matrices all of
whose $k \times k$ minors have the same sign $\epsilon_k$, for each $k$.
Unlike TP matrices, for no sign pattern $\epsilon =
(\epsilon_1, \dots, \epsilon_{\min(m,n)})$ other than $(1,\dots,1)$ and
$(-1,1,\dots,(-1)^{\min(m,n)})$, is a factorization as above known.
But this question also leads back to its counterpart for symmetric
matrices, and suggests a partitioning of all LPM matrices: via the signs
of their principal minors.\medskip

Our second, modern motivation comes from the tremendous activity around
the Cholesky and LDU decompositions -- theoretically, numerically, and in
downstream applications.
There has been significant recent activity in sampling and analyzing
data that lives on hyperbolic manifolds (and not Euclidean spaces). Such
problems have attracted researchers in 
geographic routing \cite{Kleinberg},
image and language processing \cite{DGRS,Kh-IEEE},
social networks \cite{Verbeek},
and finance \cite{KRN},
to list a few areas and papers.
In these applications, one studies the Lobachevskian analogues of
correlation matrices, termed ``Lorentz--Gram matrices'' (see
Section~\ref{Sinertia} for more details). These have negative
eigenvalues/inertia, and so for applications it is of interest to provide
a mathematical framework which will enable probability and statistics
tools over cones of such matrices. We do so in this work.

Additionally, the PD cone is homeomorphic to an $\R$-vector space via
taking logarithms~\cite{VectorSpace}. Our ``algorithmic'' factorization,
alternate to LDU, is crucial in uncovering a ``log-Cholesky'' alternative
below to this homeomorphism as well. This latter has the advantage that
it is a \textit{linear} map that is a Euclidean space isomorphism onto
$PD_n$, and thus onto every $LPM_n(\epsilon)$ cone.

\subsection{LPM matrices and their Cholesky factorization}\label{Sublpm}

Motivated by the SSR matrix cones, we begin this work by ``generalizing''
the cone $PD_n$ of positive definite $n \times n$ matrices, which is of
tremendous importance in theoretical and applied fields. The cone $PD_n$
is defined in several equivalent ways, of which we list three:
\begin{enumerate}
\item All principal minors are positive.
\item All leading principal minors are positive.
\item All trailing principal minors are positive.
\end{enumerate}

Akin to Theorem~\ref{TBFZ}, $PD_n$ admits a (celebrated)
factorization-diffeomorphism onto an orthant:

\begin{theorem}[Cholesky decomposition]\label{Tpd}
Let ${\bf L}_n$ denote the {\em Cholesky space} of lower triangular
matrices in $\R^{n \times n}$ with positive diagonal entries. Then $\Phi
: L \mapsto L L^T$ is a smooth diffeomorphism~\cite{Cholesky} between
$PD_n$ and ${\bf L}_n$, which takes Borel sets to Borel sets (both inside
$\R^{n \times n}$). Thus $PD_n$ is diffeomorphic to an open Euclidean
ball in $\R^{n(n+1)/2}$.
\end{theorem}

The final sentence -- and the ``orthant'' claim -- are because we have
the smooth diffeomorphism $\psi := (\log, \dots, \log; {\rm id}, \dots,
{\rm id}) : (0,\infty)^{n(n+1)/2} \to {\bf L}_n$, and hence the
composition (via~\eqref{Epsi})
\[
B_{\R^{n(n+1)/2}} ({\bf 0}_{n(n+1)/2},1)\overset{\Psi}{\longrightarrow}
(0,\infty)^{n(n+1)/2} \overset{\psi}{\longrightarrow} {\bf L}_n
\overset{\Phi}{\longrightarrow} PD_n.
\]

The cone of SSR matrices (see the paragraph following~\eqref{Epsi}) leads
to
(i)~trying to define ``sign pattern'' variants of real symmetric
matrices, e.g.\ via the above criteria for $PD_n$;
and then (ii)~trying to extend the LDU factorization for such matrices
\cite[Corollary 3.5.6]{HJ} to a family of such factorizations.
Interestingly, it is not description~(1), but~(2) and~(3) that are the
way to proceed here. We begin by extending description~(2) to other sign
patterns:

\begin{defn}[A novel cone]\label{Dlpm}
Given an integer $n \geq 1$ and a sign pattern $\epsilon \in \{ \pm 1
\}^n$, a real symmetric matrix $A_{n \times n}$ is said to be
$LPM_n(\epsilon)$ (\textit{Leading Principal Minors} with sign pattern
$\epsilon$) if for all $1 \leq k \leq n$, the leading principal $k \times
k$ minor of $A$ is nonzero of sign $\epsilon_k$. We say $A_{n \times n}$
is \textit{LPM} if $A$ is $LPM_n(\epsilon)$ for some $\epsilon \in \{ \pm
1 \}^n$:
\begin{equation}
LPM_n = \bigsqcup_{\epsilon \in \{ \pm 1 \}^n} LPM_n(\epsilon), \qquad n
\geq 1.
\end{equation}
\end{defn}

A distinguished family of LPM matrices is the positive definite cone
$PD_n = LPM_n({\bf 1}_n)$.
LPM matrices also include symmetric P-, N-, PN-, almost P-, and almost N-
matrices, which are widely studied in economics and game theory,
mathematical programming, complexity theory, the theory of (global)
univalence of maps, and interval matrices among others. See
Remark~\ref{Rexamples}. On the theoretical side, larger classes
containing various $LPM_n(\epsilon)$ were recently studied
in~\cite{Signs,Descartes} in broader frameworks and with different
objectives.

Note that the set $LPM_n$ is a strict subset of invertible real
symmetric matrices (for instance, it does not contain the anti-diagonal
permutation matrix $P_n$ for any $n \geq 2$ -- see Theorem~\ref{Ttpm}).
Nevertheless, we show below that it is open and dense in all symmetric
matrices. Moreover, the set $LPM_n(\epsilon)$ is nonempty for all
$\epsilon$; for instance,
\begin{equation}\label{Elpmdiag}
\D_\epsilon := \begin{pmatrix}
\epsilon_1 & 0 & 0 & \cdots & 0\\
0 & \epsilon_1 \epsilon_2 & 0 & \cdots & 0\\
0 & 0 & \epsilon_2 \epsilon_3 & \cdots & 0\\
\vdots & \vdots & \vdots & \ddots & \vdots\\
0 & 0 & 0 & \cdots & \epsilon_{n-1} \epsilon_n
\end{pmatrix}
\end{equation}
is the unique diagonal matrix in $LPM_n(\epsilon)$ with unit-modulus
entries. This matrix is often used below, notably when equipping each
$LPM_n(\epsilon)$ with a Riemannian metric (Remark~\ref{Runique}) and
with probability densities (Theorem~\ref{Twishart} and
Proposition~\ref{Pjacobian}).

We now proceed. As Theorem~\ref{Tpd} says, $PD_n$ is homeomorphic, even
diffeomorphic, to an open Euclidean ball via the Cholesky decomposition.
It is thus natural to ask if
\begin{enumerate}
\item[(a)] more Cholesky-type LDU factorizations of positive matrices
exist -- which are moreover smooth diffeomorphisms; and the same question
more generally for every cone $LPM_n(\epsilon)$.
\item[(b)] If yes, how one would go (smoothly) from any of these sets to
any other.
\end{enumerate}

The following result shows that both questions have positive answers.

\begin{utheorem}[Algorithmic and generalized Cholesky decomposition]\label{Tlpm}
Fix an integer $n \geq 1$.
\begin{enumerate}[$(1)$]
\item The sets $LPM_n(\epsilon)$, $\epsilon \in \{ \pm 1 \}^n$ are
nonempty and pairwise disjoint; and their union $LPM_n$ is open and dense
in all real symmetric matrices. More strongly, given any $A = A^T \in
\R^{n \times n}$, there exists a sign pattern $\epsilon$ and a sequence
$A_m \in LPM_n(\epsilon)$ such that $A_m \to A$ entrywise.

\item Given any sign pattern $\epsilon \in \{ \pm 1 \}^n$, and a matrix
$B_\epsilon \in LPM_n(\epsilon)$, the map
\[
\Phi_{B_\epsilon} : {\bf L}_n \to LPM_n(\epsilon) \quad \text{defined
by} \quad L \mapsto L \cdot B_\epsilon \cdot L^T
\]
is a smooth diffeomorphism such that the Cholesky decomposition
$\Phi^{-1}_{B_\epsilon}$ is algorithmic, hence preserves the families of
Borel and Lebesgue sets.
\end{enumerate}
\end{utheorem}

\begin{remark}
We stress that our proof of Theorem~\ref{Tlpm} will
(a)~show that each Cholesky-type factorization $\Phi_{B_\epsilon}^{-1} :
L B_\epsilon L^T \mapsto L$ is a diffeomorphism; and
(b)~provide an explicit algorithm to write down $\Phi_{B_\epsilon}^{-1}$.
The latter should have consequences on the numerical side and in
applications.
\end{remark}

\begin{remark}
The LDU factorization sends all $A \in LPM_n(\epsilon)$ to $L \cdot D
\cdot U$, where the matrix $D$ has diagonal entries the ratios of
successive leading principal minors of $A$. In contrast, our algorithmic
recipe inserts (any) \textit{fixed} matrix $B_\epsilon$ in the middle --
e.g.\ one may take the diagonal matrix $\D_\epsilon$ -- which is
independent of choice of $A$. As a result, our algorithmic
factorization-diffeomorphism necessarily differs from the LDU
factorization. Furthermore, when $\epsilon = {\bf 1}_n$ and $B_\epsilon =
\D_{{\bf 1}_n} = \Id_n$ -- so that the ``$D$'' is fixed and independent
of $A$ -- the factorization specializes to the usual Cholesky
factorization over $PD_n$. Hence we have used ``Cholesky'' in naming and
describing our maps.
\end{remark}

We record some more ramifications of Theorem~\ref{Tlpm}. First,
$\Phi_{B_\epsilon}$ is indeed a twofold extension of the ``usual''
Cholesky decomposition: it firstly holds for every sign pattern $\epsilon
\in \{ \pm 1 \}^n$. Additionally, for each of these cones -- including
the positive definite cone for $\epsilon = {\bf 1}_n$ -- we provide ${\bf
L}_n$-many Cholesky-type decompositions $\Phi_{B_\epsilon}$, one for
\textit{every} $B_\epsilon$ -- and all of them are smooth diffeomorphisms
(not just bijections or homeomorphisms).

Second, one can move smoothly between the nonempty sets
$LPM_n(\epsilon)$ in multiple ways: either via their diffeomorphisms to
the Euclidean ball $B_{\R^{n(n+1)/2}}({\bf 0}, 1)$; or by choosing any
$B_\epsilon \in LPM_n(\epsilon)$ for each such nonempty set, and then
using
\begin{equation}\label{Ecomposition}
\Phi_{B_{\epsilon'}} \circ \Phi_{B_\epsilon}^{-1} :
LPM_n(\epsilon) \longrightarrow {\bf L}_n \longrightarrow
LPM_n(\epsilon').
\end{equation}
This map is a composition of two smooth diffeomorphisms. (For a third
way, see Remark~\ref{Rsanity}.)

\begin{remark}
Note by Theorem~\ref{Tlpm} that one ``knows'' the set $LPM_n(\epsilon)$
for any $\epsilon$. Namely, upon fixing an arbitrary matrix $B_\epsilon
\in LPM_n(\epsilon)$ (e.g.~\eqref{Elpmdiag}), the $LPM_n(\epsilon)$
matrices are precisely $L B_\epsilon L^T$, and this is a bijection -- in
fact, smooth diffeomorphism -- onto the space ${\bf L}_n$ of all $L$.
\end{remark}

We conclude this part by noting that ``modified'' Cholesky factorizations
of non-positive definite matrices have previously appeared in the
literature. For instance, Gill and Murray \cite{GM} introduced a
factorization $A + E = LDL^T$ with $L$ lower triangular and $D$ diagonal,
in the context of Newton-type methods in numerical optimization. An
alternative modified Cholesky algorithm was proposed by Cheng--Higham
\cite{CH}, in which $A$ is indefinite but $A+E$ is positive definite for
some perturbation $E = E^T$, and $A = (P L) D (P L)^T$ for some lower
unitriangular $L$, diagonal $D$, and permutation matrix $P$. (See
also~\cite{Higham}.) Our results differ in spirit as well as in explicit
form from these works.

\subsection{Cholesky factorizations via trailing minors}

We discussed above what happens when one considers leading principal
minors. We will discuss the other two alternatives -- from the list at
the start of Section~\ref{Sublpm} -- beginning with option~(3).
(Option~(1) is discussed in Appendix~\ref{Sssrpm}.)

\begin{defn}
Given an integer $n \geq 1$ and a sign pattern $\epsilon \in \{ \pm 1
\}^n$, a real symmetric matrix $A_{n \times n}$ is said to be
$TPM_n(\epsilon)$ (\textit{Trailing Principal Minors} with sign pattern
$\epsilon$) if for all $1 \leq k \leq n$, the trailing principal $k
\times k$ minor of $A$ is nonzero of sign $\epsilon_k$. We say $A_{n
\times n}$ is TPM if $A$ is $TPM_n(\epsilon)$ for some $\epsilon \in \{
\pm 1 \}^n$:
\begin{equation}
TPM_n = \bigsqcup_{\epsilon \in \{ \pm 1 \}^n} TPM_n(\epsilon), \qquad n
\geq 1.
\end{equation}
\end{defn}

Akin to Theorem~\ref{Tlpm}, we record the analogous results for TPM
matrices. While this can be done ``from first principles'' in direct
analogy to LPM matrices, our proof will go through using either a linear
or a nonlinear diffeomorphism to $TPM_n(\epsilon)$ from
$LPM_n(\epsilon')$, for some sign patterns $\epsilon'$.

\begin{theorem}[TPM cones and the reversal map]\label{Ttpm}
Theorem~\ref{Tlpm}(1) goes through verbatim for all $TPM_n(\epsilon)$.
Moreover, given $n \geq 1$ and $\epsilon \in \{ \pm 1 \}^n$, and a matrix
$C_\epsilon \in TPM_n(\epsilon)$, the map
\[
\Phi^{C_\epsilon} : {\bf L}_n \to TPM_n(\epsilon) \quad \text{defined by}
\quad L \mapsto L^T \cdot C_\epsilon \cdot L
\]
is a smooth diffeomorphism such that the ``reverse Cholesky
decomposition'' $(\Phi^{C_\epsilon})^{-1}$ is algorithmic.
Hence it too preserves the families of Borel and Lebesgue sets.

Also define $P_n \in \R^{n \times n}$ to be the anti-diagonal permutation
matrix with $(u,v)$ entry $1$ if $u+v=n+1$, and $0$ otherwise; and define
the ``reversal'' of a square matrix to be the anti-involution
\begin{equation}\label{Ereversal}
\overset{\leftarrow}{(\cdot)} : \C^{n \times n} \to \C^{n \times n}, \qquad
A \mapsto (P_n A P_n)^*.
\end{equation}

This map is a linear (smooth) diffeomorphism $: LPM_n(\epsilon)
\longleftrightarrow TPM_n(\epsilon)$. Moreover, it yields a commuting
square of (reversible) smooth diffeomorphisms for every $B_\epsilon \in
LPM_n(\epsilon)$:
\begin{equation}\label{Erevsq}
\begin{CD}
L \in {\bf L}_n @>>> \rev{L} = (P_n L P_n)^T \in {\bf L}_n\\
@V\Phi_{B_\epsilon}VV   @V\Phi^{\rev{B_\epsilon}}VV \\
A = L B_\epsilon L^T \in LPM_n(\epsilon) @>>>
\rev{A} = \rev{L}^T \rev{B_\epsilon} \rev{L} \in TPM_n(\epsilon)
\end{CD}
\end{equation}
\end{theorem}

The commuting square~\eqref{Erevsq} will be useful later, in defining and
relating Wishart densities on the ``dual'' cones $LPM_n(\epsilon)$ and
$TPM_n(\epsilon)$ for all $n,\epsilon$.

\begin{remark}
For brevity, we will write ``diffeomorphism'' to denote
``smooth/$C^\infty$ diffeomorphism''.
\end{remark}

\subsection{Riemannian geometry (and mean)}

As is well-known, the positive definite cone $PD_n$ is also convex, which
implies that it admits the Euclidean metric and Euclidean geodesics (line
segments) -- so the Euclidean mean of $A,B$ is $(A+B)/2$. Alternately,
$PD_n$ is a Riemannian manifold, where the geometric mean $A \# B =
\gamma_{A,B}(1/2)$ \cite{ALM,Bhatia,Moakher,PW} aligns with a natural
definition of geodesics in $PD_n$, from $A$ to $B$:
\[
\gamma_{A,B} : [0,1] \to PD_n, \quad \text{defined by} \quad
\gamma_{A,B}(t) := A^{1/2} (A^{-1/2} B A^{-1/2})^t A^{1/2}.
\]
Recently, Lin~\cite{Cholesky} introduced a different Riemannian manifold
structure on $PD_n$, via the Cholesky decomposition. It is this structure
that is of interest in the current work.

As we will explain below, not all cones $LPM_n(\epsilon)$ are similarly
convex; thus, it is natural to seek a Riemannian metric. The following
result records such a metric's existence and some properties.

\begin{utheorem}[Lie group structure]\label{Triemannian}
For every $n \geq 1$ and sign pattern $\epsilon \in \{ \pm 1 \}^n$, the
cone $LPM_n(\epsilon)$ is both a Riemannian manifold (with sectional
curvature zero) as well as an abelian Lie group, under whose action the
Riemannian metric is translation (bi-)invariant. In particular, the
log-Cholesky mean/barycentre of $A = L \D_\epsilon L^T, A' = K
\D_\epsilon K^T \in LPM_n(\epsilon)$ (with $L,K \in {\bf L}_n$) is
\[
A \circledast^{1/2} A' := (L \circledcirc^{1/2} K) \D_\epsilon
(L \circledcirc^{1/2} K)^T, \quad \text{where} \quad 
L \circledcirc^{1/2} K = \begin{cases}
(l_{ij} + k_{ij})/2, \quad & \text{if } i \neq j,\\
\sqrt{l_{jj} k_{jj}}, & \text{if } i=j.
\end{cases}
\]
The analogous results hold over every cone $TPM_n(\epsilon)$.
\end{utheorem}

\subsection{Hilbert space towers; Hermitian matrices}

One may wonder if there is a familiar model for the above (isomorphic)
abelian metric groups $LPM_n(\epsilon)$, $TPM_n(\epsilon)$, and ${\bf
L}_n$ -- perhaps including the translation-invariant metric too. Our next
result shows that this abelian metric group is isometrically isomorphic
to a flat, normed space -- in fact, to a Hilbert space:

\begin{utheorem}[Hilbert space structure]\label{Teuclidean}
The abelian metric Lie groups (from Theorem~\ref{Triemannian}) ${\bf
L}_n$ and all $LPM_n(\epsilon), TPM_n(\epsilon)$ are isometrically
isomorphic to the Euclidean space $\mathbb{R}^{n(n+1)/2}$ -- and hence
separable and complete.
\end{utheorem}

Thus in addition to real analysis, the powerful and extensive machinery
of ``finite-dimensional Euclidean space probability'' immediately applies
to all of these groups. Moreover, as these are groups (in fact Hilbert
spaces) of matrices -- but with a different group operation than usual
matrix addition or multiplication -- it may be of interest to explore
random matrix theory (e.g.~\cite{FK}) on the LPM and TPM spaces, as well
as connections to geometry and to ergodic theory,
e.g.~\cite{HS,Pollicott}.

\begin{remark}
Note that the cone $PD_n$ is convex, and so admits flat geodesics (line
segments) in Euclidean space. It was also asserted
in~\cite[pp.~1355]{Cholesky} that Cholesky space ${\bf L}_n$ and $PD_n$
are not Riemannian submanifolds of a Euclidean space.
But Theorem~\ref{Teuclidean} shows that actually, they \textit{are}
indeed such submanifolds -- to be precise, they are (isomorphic to) all
of the finite-dimensional Hilbert space $\R^{n(n+1)/2}$. More generally,
so is $LPM_n(\epsilon)$ for every $\epsilon$, even if it is not always
convex (e.g.\ see Example~\ref{Exnotconvex} for $\epsilon_2 = -1$).

The point is that $PD_n$ is a dense convex cone inside the positive
semidefinite matrices -- in the \textit{usual} Euclidean norm. But this
is not (equivalent to) the metric on $LPM_n(\epsilon)$ for any $\epsilon
\in \{ \pm 1 \}^n$. For instance, the sequence $\{ \frac{1}{k}
\D_\epsilon : k \geq 1 \}$ is unbounded in the metric on
$LPM_n(\epsilon)$, whereas $\frac{1}{k} \Id_n$ is bounded, even
Cauchy in the Euclidean norm on $PD_n$.
\end{remark}

Continuing along Hilbertian lines: notice the tower of Euclidean spaces
$\R^{\binom{2}{2}} \subset \R^{\binom{3}{2}} \subset \cdots$. This is
reflected in the Cholesky spaces and LPM-cones: given a sign sequence
$\beps := (\epsilon_1, \epsilon_2, \dots) \in \{ \pm 1 \}^\infty$,
\[
{\bf L}_1 \cong LPM_1((\epsilon_1)) \quad \hookrightarrow \quad {\bf L}_2
\cong LPM_2((\epsilon_1, \epsilon_2)) \quad \hookrightarrow \quad \cdots.
\]
These embeddings form a commuting square for each $n \geq 1$:
\begin{equation}\label{Esquare}
   \begin{CD}
     L \in {\bf L}_n @>>> L' := \begin{pmatrix} L & {\bf 0}_n \\ {\bf 0}_n^T & 1 \end{pmatrix} \in {\bf L}_{n+1}\\
     @V\Phi_{\D_\epsilon}VV   @V\Phi_{\D_{\epsilon'}}VV \\
     A = L \D_\epsilon L^T \in LPM_n(\epsilon) @>>>
     A' := L' \D_{\epsilon'} (L')^T \in LPM_{n+1}(\epsilon')
   \end{CD}
\end{equation}
where $\epsilon = (\epsilon_1, \dots, \epsilon_n)$,
$\epsilon' = (\epsilon_1, \dots, \epsilon_n, \epsilon_{n+1})$,
and
\[
A' := \Phi_{\D_{\epsilon'}}(L') = \displaystyle \begin{pmatrix} A & {\bf
0}_n \\ {\bf 0}_n^T & \epsilon_n \epsilon_{n+1} \end{pmatrix} =
\begin{pmatrix} L & {\bf 0}_n \\ {\bf 0}_n^T & 1 \end{pmatrix} \cdot
\begin{pmatrix} \D_\epsilon & {\bf 0}_n \\ {\bf 0}_n^T & \epsilon_n
\epsilon_{n+1} \end{pmatrix} \cdot \begin{pmatrix} L & {\bf 0}_n \\ {\bf
0}_n^T & 1 \end{pmatrix}^T.
\]

This tower of Euclidean spaces (via Theorem~\ref{Teuclidean}) has a union
/ direct limit, which in turn has a closure. We now identify model inner
product spaces for both of these.

\begin{utheorem}[Cholesky space towers]\label{Thilbert}
Fix a sign sequence $\beps = (\epsilon_1, \dots, \epsilon_n, \dots) \in
\{ \pm 1 \}^\infty$.
\begin{enumerate}[$(1)$]
\item Let ${\bf L}_{00}$ denote the set of semi-infinite real lower
triangular matrices with
(i)~all diagonal entries in~$(0,\infty)$,
(ii)~all but finitely many diagonal entries $1$, and
(iii)~all but finitely many entries below the diagonal zero.
Then ${\bf L}_{00}$ is the direct limit of the Euclidean spaces ${\bf
L}_n$ under the inclusions~\eqref{Esquare}, and is isomorphic to the real
inner product space $c_{00}$.

\item Analogously denote by $LPM_{00}(\beps)$ the semi-infinite real
symmetric matrices $\begin{pmatrix} A_{n \times n} & {\bf 0} \\ {\bf 0} &
\D_{\beps'} \end{pmatrix}$,
where $n \geq 1$, $A \in LPM_n((\epsilon_1, \dots, \epsilon_n))$, and
$\beps' = (\epsilon_n \epsilon_{n+1}; \, \epsilon_n \epsilon_{n+2},
\epsilon_{n+1} \epsilon_{n+3}, \dots)$.
Thus $LPM_{00}(\beps)$ is the direct limit of the Euclidean spaces
$LPM_n((\epsilon_1, \dots, \epsilon_n))$ under~\eqref{Esquare}, hence
$LPM_{00}(\beps) = \Phi_{\D_{\beps}}({\bf L}_{00}) \cong c_{00}$, where
$\Phi_{\D_{\beps}}(L) := L \D_{\beps} L^T$. 

\item Let ${\bf L}_{\mathcal{H}} \supset {\bf L}_{00}$ denote the
semi-infinite real lower triangular matrices $L = (l_{ij})_{i,j \geq 1}$,
with
(i)~all diagonal entries in $(0,\infty)$,
(ii)~$\sum_{j\geq 1} (\log l_{jj})^2 < \infty$, and
(iii)~$\sum_{i>j \geq 1} |l_{ij}|^2 < \infty$.
Then ${\bf L}_{\mathcal{H}}$ (and its transfer under $\Phi_{\D_{\beps}}$)
is isomorphic to $\ell^2_\R = \overline{c_{00}}$ as a real Hilbert space.
\end{enumerate}
\end{utheorem}

Thus, the Euclidean spaces $LPM_n((\epsilon_1, \dots, \epsilon_n))$ form
a nested sequence of based spaces (Riemannian manifolds), whose union
$LPM_{00}(\beps)$ is again a real inner product space. In addition, the
map $\Phi^{-1}$ on the image $\Phi({\bf L}_{\mathcal{H}})$ amounts to
\textit{semi-infinite Cholesky decomposition}.

We next turn to complex analogues of the above results.
These go through without any surprises.

\begin{utheorem}[Complex matrices]\label{Thermitian}
Fix an integer $n \geq 1$ and a sign pattern $\epsilon \in \{ \pm 1
\}^n$. Let ${\bf L}^\C_n$ denote the lower triangular matrices in $\C^{n
\times n}$ with positive real diagonals.
Now define $LPM^\C_n(\epsilon)$ to be the subset of Hermitian matrices $A
\in \C^{n \times n}$ whose leading principal $k \times k$ minor has sign
$\epsilon_k$ for all $1 \leq k \leq n$.
\begin{enumerate}[$(1)$]
\item The sets $LPM_n^\C(\epsilon)$ are all nonempty and pairwise
disjoint; and their union over all $\epsilon$ is open and dense in the $n
\times n$ Hermitian matrices.

\item Given any $B_\epsilon \in LPM_n^\C(\epsilon)$, the map
$\Phi^\C_{B_\epsilon} : {\bf L}^\C_n \to LPM^\C_n(\epsilon)$,
defined by $L \mapsto L \cdot B_\epsilon \cdot L^*$,
is a smooth (real) diffeomorphism with $(\Phi^\C_{B_\epsilon})^{-1}$
algorithmic.

\item The natural analogue of Theorem~\ref{Ttpm} also goes through for
$TPM_n^\C$ (defined similarly).

\item The analogues of Theorems~\ref{Teuclidean} and~\ref{Thilbert} go
through in the complex setting; the cone ${\bf L}_n^\C$ is isomorphic to
the {\em real} Euclidean space $\mathbb{R}^{n^2}$, while ${\bf
L}_{00}^\C$ and $LPM_{00}^\C(\beps)$ are isomorphic to $c_{00}^\R \oplus
c_{00}^\C$ (for the diagonal and off-diagonal elements, respectively) as
real inner product spaces. Similarly, ${\bf L}^\C_{\mathcal{H}}$ is
(linearly and isometrically) isomorphic to $\ell^2_\R \oplus \ell^2_\C$
as a real Hilbert space.
\end{enumerate}
\end{utheorem}

We also add that since ${\bf L}^\C_n$ is smoothly diffeomorphic to an
open real Euclidean ball, the remarks made after Theorem~\ref{Tlpm} are
equally valid for TPM matrices and for the complex variants of both.

\subsection{The cone $LPM_n$ is a complete abelian metric group;
stochastic inequalities}

The above sections showed that the cone $LPM_n$ is partitioned into $2^n$
pairwise diffeomorphic sub-cones $LPM_n(\epsilon)$ (and similarly
for $TPM_n$), each of which is
(a)~an abelian metric group (in fact isomorphic to a Euclidean space);
as well as
(b)~a Riemannian manifold with a translation-invariant metric. Thus, each
sub-cone is ``parallel'' to $PD_n$. However, a drawback is that this
model does not determine distances between elements of distinct cones
$LPM_n(\epsilon)$.

We show below that under a -- very similar but subtly different -- binary
operation, the entire cone $LPM_n$ can be equipped with structures~(a)
and~(b), such that
(b$'$)~every cone $LPM_n(\epsilon)$ isometrically embeds as a
submanifold, and
(a$'$)~$PD_n$ remains a subgroup.

\begin{utheorem}[Larger Lie groups]\label{Tbiggroup}
The cones $LPM_n$, $TPM_n$, and their Hermitian counterparts can each be
given the structure of a complete separable abelian group and a
Riemannian manifold -- with a bi-invariant metric -- such that $PD_n$,
together with the log-Cholesky metric, is the identity-component subgroup
of index $2^n$. However, none of these groups embed in any Banach space.
\end{utheorem}

Theorem~\ref{Tbiggroup} implies that while Euclidean space probability
(or even Banach space probability~\cite{LT}) can not apply to $LPM_n$,
many fundamental stochastic inequalities hold nevertheless:

\begin{utheorem}[Probability]\label{Tprob}
Stochastic inequalities by Hoffman-J{\o}rgensen, L\'evy--Ottaviani,
Mogul'skii, and Ottaviani--Skorohod hold over each cone in
Theorem~\ref{Tbiggroup} (and over other subfields of $\C$).
\end{utheorem}

We elaborate on this theme in Section~\ref{Sprob}.

\begin{remark}
Notice that the Riemannian metric in Theorem~\ref{Tbiggroup} restricts to
the sub-cones $LPM_n(\epsilon), TPM_n(\epsilon)$ -- and as we will show,
agrees there with the metric in Theorem~\ref{Triemannian} -- for each
$\epsilon \in \{ \pm 1 \}^n$. However, these cones are all nontrivial
cosets of $PD_n$ for $\epsilon \neq {\bf 1}_n$, hence not subgroups of
$LPM_n$ or $TPM_n$. Viewed ``dually'', the group operation on
$LPM_n(\epsilon)$ in Theorem~\ref{Triemannian} differs from that in
Theorem~\ref{Tbiggroup} for every $\epsilon \neq {\bf 1}_n$.

That said, certainly the inequalities in Theorem~\ref{Tprob} also hold
over each individual cone $LPM_n(\epsilon)$, $TPM_n(\epsilon)$ under the
metric in Theorems~\ref{Triemannian} and~\ref{Teuclidean}. We omitted
saying this because far more (i.e., all of finite-dimensional
probability) holds on these individual cones.
\end{remark}

\subsection{Wishart and other densities on LPM and TPM cones}

Finally, we come to distributions used in multivariate (statistical)
analysis and random matrix theory on the positive definite cone. As a
concrete/working example, we will transfer the (inverse) Wishart and
lognormal probability densities from the PD cone to every
$LPM_n(\epsilon)$ and $TPM_n(\epsilon)$ cone, via the following recipe.

\begin{defn}\label{Dwishart}
Let $n \geq 1$ and let $Q$ be a probability distribution with support in
$PD_n$ and density function $f_Q(\cdot)$. Given $\epsilon \in \{ \pm 1
\}^n$, define the ``transfer'' distributions $Q_\epsilon^{LPM},
Q_\epsilon^{TPM}$ via:
\begin{align}
&\ f_{\epsilon,Q}^{LPM} : LPM_n(\epsilon) \to [0,\infty), \qquad
{\bf M} = L \D_\epsilon L^T \mapsto f_Q(L L^T)\\
\text{and} \quad
&\ f_{\epsilon,Q}^{TPM} : TPM_n(\epsilon) \to [0,\infty), \qquad
\rev{\bf M} = K^T \D_{\epsrev} K \mapsto f_Q(K^T K).
\end{align}
\end{defn}

These make sense given the bijections $\Phi_{\D_\epsilon},
\Phi^{\D_\epsilon}$ from Theorems~\ref{Tlpm} and~\ref{Ttpm}. The point
there was that these maps are in fact nicer topologically -- being smooth
diffeomorphisms -- which has the ``probability consequence'' that they
preserve the Borel $\sigma$-algebras in both sub-cones of $\mathbb{R}^{n
\times n}$ -- even the Lebesgue $\sigma$-algebras, being continuous.
But even more holds:

\begin{utheorem}[Random matrix theory]\label{Twishart}
Let $Q$ be the probability distribution of a continuous random variable
on $PD_n$ (with the Lebesgue $\sigma$-algebra).
\begin{enumerate}[$(1)$]
\item For all $\epsilon \in \{ \pm 1 \}^n$, the transfer probabilities
under $Q_\epsilon^{LPM}, Q_\epsilon^{TPM}$ equal those under $Q$ -- i.e.,
for all events/measurable subsets $\mathscr{A} \subseteq
LPM_n(\epsilon)$, we have
\begin{equation}\label{Echangeofvar}
\bbp_{\epsilon,Q}^{LPM}( {\bf M} = L \D_\epsilon L^T \in \mathscr{A} ) =
\bbp_Q \left( \Phi_{\Id_n} \circ \Phi_{\D_\epsilon}^{-1} ( {\bf M} ) = L
L^T \in \Phi_{\Id_n} \circ \Phi_{\D_\epsilon}^{-1} (\mathscr{A}) \right).
\end{equation}
This uses that at any point in ${\bf L}_n$, the Jacobian of
$\Phi_{\D_\epsilon}(L)$ with respect to $L$ is lower triangular, and the
absolute value of its determinant is $2^n \prod_{j=1}^n l_{jj}^{n+1-j}$.
(See Proposition~\ref{Pjacobian} for details.)

\item Let $Q = W_n(\Sigma,N)$ be any Wishart distribution -- where
$\Sigma \in PD_n$ and $N \geq n \geq 1$. 
Write down the Cholesky decomposition of $\Sigma$ (as above) as
$L_\circ L_\circ^T$, and the Bartlett decomposition of ${\bf M_1}
\sim W_n(\Sigma,N)$ as ${\bf M_1} = L_\circ K_1 K_1^T L_\circ^T$. Then
the {\em signed Bartlett decomposition}, defined via ${\bf M} := L_\circ
K_1 \D_\epsilon K_1^T L_\circ^T$, has the Wishart distribution
$W_{\epsilon,n}^{LPM}(\Sigma,N)$.

\item Similar results hold for $\rev{\bf M} \sim Q_\epsilon^{TPM}$.

\item For all integers $N \geq n \geq 1$ and matrices $\Sigma \in PD_n$,
we have
\begin{equation}\label{Ewishartrev}
{\bf M} \sim W_{\epsilon,n}^{LPM}(\Sigma,N) \qquad \Longleftrightarrow
\qquad \rev{\bf M} \sim W_{\epsilon,n}^{TPM}(\rev{\Sigma}, N).
\end{equation}

\item Suppose $g : \R \to \R$ is a Lebesgue measurable function,
and~\eqref{Ewishartrev} holds. Then $\mathbb{E}[g(\rev{\bf M})] =
\overset{\longleftarrow}{\mathbb{E}}[g({\bf M})]$, if either side exists.
\end{enumerate}
\end{utheorem}

\begin{remark}
Thus, not only does Theorem~\ref{Twishart} provide novel Wishart-type
densities for random matrices in the LPM and TPM cones, we also provide
random variables ``realizing'' these densities via \textit{signed
Bartlett decompositions}. Correspondingly -- via the reversal map -- the
Wishart TPM densities are precisely those for the random variables
\begin{equation}
\rev{\bf M} = (\rev{L_\circ})^T (\rev{K_1})^T \D_{\epsrev} \rev{K_1}
\rev{L_\circ} \, \sim \, W_{\epsilon,n}^{TPM}(\rev{\Sigma},N).
\end{equation}
\end{remark}

\begin{remark}
In addition to showing Theorem~\ref{Twishart}, we introduce a novel (to
our knowledge) family of probability densities on the PD cone, which we
term \textit{Cholesky-normal densities} -- see
Definition~\ref{Dcholeskynormal}. By above, these transfer at once to
density functions on LPM and TPM cones.
\end{remark}

On a closing note, the Cholesky decomposition and Wishart distribution
have tremendous utility in theoretical and applied fields -- always on
the positive definite cone. Above, we have provided not just examples of
$LPM_n(\epsilon)$ matrices, but a complete ``enumeration'' in
Theorem~\ref{Tlpm}. In addition, we have proved a Cholesky-type
decomposition on an open dense set of all real symmetric matrices, each
of which is an algorithmic Riemannian isometry. It is hoped that these
results will open up several fronts of enquiry: numerically,
geometrically, in random matrix theory and probability, in statistical
multivariate analysis (and estimation), and in applications.

\subsection{Organization of the paper}

In Section~\ref{Slpm}, we explain why Theorem~\ref{Tlpm} holds, providing
the algorithmic Cholesky decomposition of $LPM_n$ matrices. We then show
Theorem~\ref{Ttpm} on TPM matrices in Section~\ref{Stpm}.
Next, in Section~\ref{Sriemannian} we discuss how our Cholesky-type
decomposition leads to additional rigid mathematical structure, leading
to the applicability on $LPM_n$ (and $TPM_n$) of Riemannian geometry --
in other words, we prove and elaborate on Theorem~\ref{Triemannian}.

In Section~\ref{Seuclidean}, we show Theorems~\ref{Teuclidean}
and~\ref{Thilbert}: the Cholesky and LPM cones, as well as their
union-closure, are real Hilbert spaces. We also show the complex
analogues in Theorem~\ref{Thermitian}, and then extend some of these
results to other subfields of $\R$ or $\C$ (which do not yield Hilbert
spaces).

Section~\ref{Sbiggroup} is motivated by applying probability inequalities
on LPM spaces. We first prove Theorem~\ref{Tbiggroup} about the abelian
metric group structure on all of $LPM_n$, on the direct limit $LPM_{00}$,
and on $TPM_n$ -- over various subfields of $\C$. We also study
probability results (including Theorem~\ref{Tprob}) on these ``bigger''
cones in Section~\ref{Sprob}. Then in Section~\ref{Sstat} we introduce
the (inverse) Wishart, lognormal, and other distributions supported on
$LPM_n(\epsilon)$, the open dense cone $LPM_n$, and their TPM analogues,
and prove Theorem~\ref{Twishart}. Moreover, we introduce the
Cholesky-normal density on the PD cone itself; and examine how $LPM_n$
matrices with specified inertia partition into sub-cones
$LPM_n(\epsilon)$. This enables defining probability densities on
inertial LPM cones. In addition to the broad goal of developing the
theory further, we also list a few specific future questions.

We conclude with two Appendices. 
In Appendix~\ref{Soperations}, we examine the behavior of
$LPM_n(\epsilon)$ and $TPM_n(\epsilon)$ matrices under Kronecker
products and diagonal concatenation (which we term ``direct sum'').
Appendix~\ref{Sssrpm} begins with option~(1) at the top of
Section~\ref{Sublpm} and studies a related sub-cone of matrices, all of
whose principal $k \times k$ minors have sign $\epsilon_k$.

\section{Algorithmic Cholesky decomposition for LPM matrices}\label{Slpm}

We begin with global notation. Let $[k] := \{ 1, \dots, k \}$ for an
integer $k>0$. Now given a matrix $A_{m \times n}$ and sets $J \subseteq
[m], K \subseteq [n]$, denote by $A_{JK}$ the submatrix of $A$ with rows
and columns indexed by $J,K$ respectively.

The first goal in this section is to show:

\begin{proof}[Proof of Theorem~\ref{Tlpm}]\hfill
\begin{enumerate}[(1)]
\item The sets $LPM_n(\epsilon)$ are clearly pairwise disjoint across all
$\epsilon$, and all nonempty by~\eqref{Elpmdiag}. Now set
$t_{\bf 0} := 1$, and for any $A = A^T \in \R^{n \times n} \setminus
\{ {\bf 0} \}$, set
\[
t_A := \min \{ |\lambda| : \lambda \text{ is a nonzero eigenvalue of some
leading principal submatrix of } A \}.
\]
This works for all $A \neq {\bf 0}$, since if all eigenvalues are zero
then $A={\bf 0}$ by the spectral theorem.

Now fix $k \in [1,n]$ and list the spectrum $\sigma(A_{[k][k]}) =
\{ \lambda_1 \leq \cdots \leq \lambda_k \}$.
The eigenvalues of the linear pencils of ``perturbed leading principal
submatrices'' are
\[
\sigma( A_{[k][k]} + t \, \Id_k ) = \{ \lambda_1 + t \leq \cdots \leq
\lambda_k + t \}, \qquad 0 < t < t_A.
\]
By choice of $t_A$, these eigenvalues satisfy: for every $1 \leq j \leq
k$, the real scalar $\lambda_j + t$ is nonzero, and of a constant sign
across all $0 < t < t_A$. Hence $\det (A_{[k][k]} + t \, \Id_k)$
has sign $\epsilon_k$ independent of $t \in (0,t_A)$. As this is true for
each $k$, we obtain the desired sign pattern $\epsilon = (\epsilon_1,
\dots, \epsilon_n)$. Now $A + t\, \Id_n \in LPM_n(\epsilon)$ for $t \in
(0,t_A)$, and we can set e.g.\ $A_m = A + \frac{t_A}{m+1} \Id_n$ for $m
\geq 1$.

This shows density; moreover, the complement of $LPM_n$ is the union of
the $n$ hypersurfaces in $\mathbb{R}^{n(n+1)/2}$ given by the vanishing
of the $n$ leading principal minors -- i.e., $Z(a_{11}) \cup Z(a_{11}
a_{22} - a_{12}^2) \cup \cdots \cup Z(\det(A))$. As each zero-locus set
is closed, the complement $LPM_n$ is open.

\item For any $1 \leq k \leq n$, the Cauchy--Binet formula yields:
\begin{align}\label{Ecauchybinet}
\det (L B_\epsilon L^T)_{[k] [k]} = &\ \det( L_{[k] [n]} B_\epsilon
(L_{[k] [n]})^T)\\
= &\ \sum_{\substack{J,K \subseteq [n],\\ |J|=|K|=k}} \det(L_{[k]J})
\det(B_\epsilon)_{JK} \det(L_{[k]K})^T
= (\det L_{[k][k]})^2 \det (B_\epsilon)_{[k][k]},\notag
\end{align}
because $L$ is lower triangular so $L_{[k]J}$ is singular unless $J =
[k]$. Now $B_\epsilon \in LPM_n(\epsilon)$, so $L B_\epsilon L^T \in
LPM_n(\epsilon)$.

We next show the bijectivity and smoothness of the reverse map
$\Phi_{B_\epsilon}^{-1}$ by induction on $n$, noting that if
$B_\epsilon$ realizes a sign pattern $\epsilon$, then each of its leading
principal $k \times k$ submatrices also realizes the truncated
sign-subpattern $(\epsilon_1, \dots, \epsilon_k)$, for all $1 \leq k \leq
n$.

Begin with the base case $n=2$.
Fix $B_\epsilon = \begin{pmatrix} m & u \\ u & v \end{pmatrix} \in
LPM_2(\epsilon)$ -- so $m$ has sign $\epsilon_1$ and $mv-u^2$ has
sign $\epsilon_2$. Now given $A = \begin{pmatrix} a & b \\ b & c
\end{pmatrix} \in LPM_2(\epsilon)$, we need to solve
\begin{equation}\label{E2x2}
\begin{pmatrix} a & b \\ b & c \end{pmatrix} =
\begin{pmatrix} l & 0 \\ p & q \end{pmatrix}
\begin{pmatrix} m & u \\ u & v \end{pmatrix}
\begin{pmatrix} l & p \\ 0 & q \end{pmatrix} =
\begin{pmatrix}
m l^2 & l (m p + u q) \\
l (m p + u q) & m p^2 + v q^2 + 2 u p q
\end{pmatrix}
\end{equation}
for $l,q>0$ and $p$. Clearly, $l = \sqrt{a/m}$, whence
\[
ac-b^2 = l^2 q^2 (mv - u^2) = a q^2 (mv-u^2)/m \quad \implies \quad
q = \sqrt{ \frac{ac-b^2}{a} \cdot \frac{m}{mv-u^2} }.
\]
(Note that the expression inside the square root is positive, since
$B_\epsilon,A \in LPM_2(\epsilon)$.) Both $l,q$ are positive by
assumption. From them, we can solve for $p$ as well, to collectively
obtain:
\begin{equation}\label{E2x2q}
l = \sqrt{a/m}, \qquad
q = \sqrt{ \frac{ac-b^2}{a} \cdot \frac{m}{mv-u^2} }, \qquad
p = \frac{b}{lm} - \frac{uq}{m}.
\end{equation}
As $l,p,q$ are uniquely found from $A$ (upon fixing $B_\epsilon$),
$L \longleftrightarrow A = L \begin{pmatrix} m & u \\ u & v \end{pmatrix}
L^T$ is a bijection.

Since $L \mapsto L \begin{pmatrix} m & u \\ u & v \end{pmatrix} L^T$ is a
smooth map, it remains to show that $A \mapsto (l,p,q)$ is also smooth.
First, $A \mapsto a/m$ is a smooth and positive function on
$LPM_2(\epsilon)$, so the map $A \mapsto l$ is smooth. Since
$(ac-b^2)/(mv-u^2) > 0$ on $LPM_2(\epsilon)$, $A \mapsto (l,q)$ is also
smooth. Finally, the formula for $p$ and the preceding calculations show
that $A \mapsto (l,p,q)$ (or $L$) is a smooth map.\medskip

We next show the induction step. Suppose $LPM_n(\epsilon)$ is
nonempty for some $n \geq 3$ and some $\epsilon \in \{ \pm 1 \}^n$. Fix a
matrix in it, say $B_\epsilon$, which we write in block form as
$B_\epsilon = \begin{pmatrix} M & {\bf u} \\ {\bf u}^T & v \end{pmatrix}$,
with the matrix $M \in LPM_{n-1}((\epsilon_1, \dots, \epsilon_{n-1}))$
and the column ${\bf u} \in \R^{n-1}$.

We now write the system to be solved as $A' = L' B_\epsilon (L')^T$.
Letting $L' = \begin{pmatrix} L & {\bf 0}_{n-1} \\ {\bf p}^T & q
\end{pmatrix}$ with $L \in \R^{(n-1) \times (n-1)}$ lower triangular,
${\bf p} \in \R^{n-1}$, and $q \in (0,\infty)$, we have
\begin{align*}
A' = \begin{pmatrix} A & {\bf b} \\ {\bf b}^T & c \end{pmatrix} = &\
\begin{pmatrix} L & {\bf 0} \\ {\bf p}^T & q \end{pmatrix}
\begin{pmatrix} M & {\bf u} \\ {\bf u}^T & v \end{pmatrix}
\begin{pmatrix} L^T & {\bf p} \\ {\bf 0} & q \end{pmatrix}\\
= &\ \begin{pmatrix}
L M L^T & L (M {\bf p} + q {\bf u}) \\
({\bf p}^T M + q {\bf u}^T) L^T & {\bf p}^T M {\bf p} + v q^2 + 2 q {\bf
u}^T {\bf p}
\end{pmatrix},
\end{align*}
where $A'$ is any element of $LPM_n(\epsilon)$.

By the induction hypothesis, one can solve $L M L^T = A$ smoothly (and
uniquely) in $L$. That is, $A' \to A \to L$ is a smooth map. Now equating
the $(1,2)$ or $(2,1)$ blocks,
\begin{equation}\label{Ep}
{\bf p} = M^{-1} (L^{-1} {\bf b} - q {\bf u}),
\end{equation}
and substituting this into the $(2,2)$ block yields:
\begin{align*}
c = &\ (L^{-1} {\bf b} - q {\bf u})^T M^{-1} (L^{-1} {\bf b} - q {\bf u})
+ 2 q {\bf u}^T M^{-1} (L^{-1} {\bf b} - q {\bf u}) + v q^2\\
= &\ \left( {\bf
b}^T A^{-1} {\bf b} - 2 q {\bf u}^T M^{-1} L^{-1} {\bf b} + q^2 {\bf u}^T
M^{-1} {\bf u} \right) + \left( 2 q {\bf u}^T M^{-1} L^{-1} {\bf b} - 2
q^2 {\bf u}^T M^{-1} {\bf u} \right) + v q^2.
\end{align*}
Canceling and regrouping,
\[
q^2 = \frac{c - {\bf b}^T A^{-1} {\bf b}}{v - {\bf u}^T M^{-1} {\bf u}}.
\]

But the theory of Schur complements says that $\det \begin{pmatrix} A
& {\bf b} \\ {\bf b}^T & c \end{pmatrix} = (\det A) (c - {\bf b}^T A^{-1}
{\bf b})$ whenever $A$ is invertible. Thus
\begin{equation}\label{Eqsquare}
q^2 = \frac{\det \begin{pmatrix} A & {\bf b} \\ {\bf b}^T  & c
\end{pmatrix}}{\det A} \cdot \frac{\det M}{\det \begin{pmatrix} M & {\bf
u} \\ {\bf u}^T & v \end{pmatrix}},
\end{equation}
where all scalars are nonzero and the right-hand side is positive by the
$LPM_n(\epsilon)$ hypothesis. As $\det A$ is uniformly positive or
uniformly negative, $A' \mapsto q^2$ is a smooth positive function of the
given data. As $q>0$, we obtain it uniquely -- and smoothly -- from $A'$,
by taking the positive square root of~\eqref{Eqsquare} (compare
with~\eqref{E2x2q}).

Finally $A' \mapsto A \mapsto L$ is smooth, so $\det(L)>0$ and hence
$L^{-1}$ are also smooth in $A$. Hence ${\bf p}$ is also smooth in $A'$,
via~\eqref{Ep}. Thus $A' \mapsto L'$ is a smooth bijection; since $L'
\mapsto L' B_\epsilon (L')^T$ is also smooth, we obtain the desired
diffeomorphism.
For the last line: continuous functions are Borel/Lebesgue measurable,
and $\Phi_{B_\epsilon}$ is a homeomorphism.\qedhere
\end{enumerate}
\end{proof}

The proof reveals that these general Cholesky-type decompositions are
algorithmic, which will have ramifications in downstream theory and
applications. For completeness, we write down
Algorithm~\ref{AlgCholesky} -- which is implicit in the proof above --
but using conjugate transposes in order to incorporate the case
discussed below, of entries in $\C$ or other subfields of $\C$.

\begin{algorithm}\caption{LPM-Cholesky}\label{AlgCholesky}
\begin{algorithmic}[1]
\State \textbf{Input:} integer $n \geq 1$, sign pattern $\epsilon \in \{ \pm 1 \}^n$, matrix $B_\epsilon \in LPM^\C_n(\epsilon)$.
\State \textbf{Input:} matrix $A \in LPM^\C_n(\epsilon)$, to be Cholesky-factored.
\State \textit{Return:} $l_{11} = \sqrt{a_{11} / (B_\epsilon)_{11}}$.
\For {$j=2,\ldots,n$}
	\State Record the smaller lower triangular matrix obtained: $L_{j-1} \in \C^{(j-1) \times (j-1)}$.
	\For {$i=1,\ldots,j-1$}
		\State \textit{Return:} $l_{ij} = 0$.
	\EndFor \ (At this stage, only the final row of $L_j$ remains to be computed.)
	\State Write $A_{[j][j]} = \begin{pmatrix} A' & {\bf b} \\ {\bf b}^* & c \end{pmatrix}$ and $(B_\epsilon)_{[j][j]} = \begin{pmatrix} M' & {\bf u} \\ {\bf u}^* & v \end{pmatrix}$.
	\State \textit{Return:} $l_{jj} = \sqrt{ \frac{\det A_{[j][j]}}{\det A'} \cdot \frac{\det M'}{\det (B_\epsilon)_{[j][j]}} }$.
	\State Compute ${\bf p} = (M')^{-1}(L_{j-1}^{-1} {\bf b} - l_{jj} {\bf u})$.
	\State \textit{Return:} $(l_{j1}, \dots, l_{j,j-1}) = {\bf p}^T$.
\EndFor
\end{algorithmic} 
\end{algorithm}

Having factored the matrices in $LPM_n = \bigsqcup_{\epsilon \in \{ \pm 1
\}^n} LPM_n(\epsilon)$, it is natural to ask if this factorization
extends to the closure, which is all real symmetric matrices. This would
generalize extending the Cholesky factorization from $PD_n$ to positive
semidefinite matrices. Unfortunately, this extension to
$\overline{LPM_n}$ fails even in the $n=2$ case:

\begin{example}\label{Exclosure}
Suppose $n=2$ and $A = \begin{pmatrix} 0 & b \\ b & c \end{pmatrix}$,
with $b \neq 0$. If one naively tries to use~\eqref{E2x2} to factorize
$A$, then $ml^2 = a = 0$ yields $l=0$, in which case equating the $(1,2)$
entries yields $b=l(mp+uq)=0$, which is false.

We also try to emulate the usual Cholesky strategy for factoring a
positive semidefinite matrix, wherein one perturbs this matrix by a
multiple of the identity, say $t\, \Id_2$. Here $|t|>0$ is small enough
such that $A(t') := \begin{pmatrix} t' & b \\ b & c+t' \end{pmatrix}$ has
negative determinant for $t'$ between $0$ and $t$.
Then $A(t/k) \in LPM_2(\frac{t}{|t|}, -1)$ for all $k \geq 1$.

Now we factor $A(t/k)$ for each $k$ using~\eqref{E2x2q}, as $L_k
A_\epsilon L_k^T$, where $A_\epsilon = \begin{pmatrix} t/|t| & 0 \\ 0 &
-t/|t| \end{pmatrix} \in LPM_2(\frac{t}{|t|}, -1)$. The $(2,2)$ entry of
$L_k$ is
\[
q_k = \sqrt{ \frac{b^2-t(ck+t)/k^2}{|t/k|} } = \sqrt{\frac{b^2 k}{|t|} -
\frac{t(ck+t)}{k|t|} },
\]
and this grows as $O(\sqrt{k})$ as $k \to \infty$, since $b \neq 0$. Thus
the matrices $L_k$ do not have a subsequence of convergent matrices, and
this approach also fails to factorize $A$ in the desired manner. \qed
\end{example}

\begin{remark}
Example~\ref{Exclosure} reveals an important distinction between the
positive/negative definite cones and other $LPM_n$ cones: the proof of
Cholesky factorization for \textit{singular} positive semidefinite
matrices fails to go through in the latter. One key step in this proof
that does not go through is that if $A \in \overline{PD_n} \setminus
PD_n$ and $A_k = A + \frac{1}{k} \Id_n$ is Cholesky-factored as $A_k
= L_k L_k^T$, then the operator norms of all $\| L_k \|$ are uniformly
bounded because of the strong property that in the $C^*$-algebra
$\mathcal{B}(\R^n)$, $\| L_k L_k^T \| = \| L_k \|^2$. This is
used to upper bound all $\| L_k \|$ and show that $L_k$ admits a
convergent subsequence (in the operator norm, hence entrywise).

As the calculation in a non positive/negative definite cone
$LPM_n(\epsilon)$ would involve $\| L_k B_\epsilon L_k^T \|$, the
sequence $L_k$ need not remain bounded, as is the case in
Example~\ref{Exclosure}.
\end{remark}

As we will see in Section~\ref{Sriemannian}, another approach does not
work either: trying to approximate all matrices in $\overline{LPM_n}$ --
e.g.\ ${\bf 0}_{n \times n}$ -- by a sequence $A_m$ of $LPM_n(\epsilon)$
matrices, along a Riemannian geodesic in the manifold $LPM_n(\epsilon)$.
The point is that all $LPM_n(\epsilon)$ are complete, and any such
sequence $A_m$ will be unbounded in the Riemannian metric, so this
approach cannot work.

\section{TPM matrices and their reverse-Cholesky
factorization}\label{Stpm}

In this short section, we explain Theorem~\ref{Ttpm} via some additional
structure. The key observation is that there exist both \textit{linear}
and \textit{nonlinear} diffeomorphisms between $LPM_n$ and $TPM_n$:

\begin{prop}\label{Ptpm}
Given $n \geq 1$ and a sign pattern $\epsilon \in \{ \pm 1 \}^n$, define
its ``reversal'' to be
\[
\epsrev := (\epsilon_n \epsilon_{n-1}, \epsilon_n \epsilon_{n-2}, \dots,
\epsilon_n \epsilon_1, \epsilon_n), \qquad \forall n \geq 2,
\]
and $\epsrev := (\epsilon_1)$ if $n=1$.
Now let $P_n$ be as in Theorem~\ref{Ttpm}. Then the maps
\begin{equation}\label{Ediffeos}
A \mapsto P_n A P_n \qquad \text{and} \qquad A \mapsto A^{-1}
\end{equation}
commute, and are linear and nonlinear smooth diffeomorphisms of order-2
between $TPM_n \longleftrightarrow LPM_n$. The first sends
$LPM_n(\epsilon) \longleftrightarrow TPM_n(\epsilon)$ and the second
sends $LPM_n(\epsilon) \longleftrightarrow TPM_n(\epsrev)$.
\end{prop}

\begin{remark}\label{Rsanity}
As a sanity check, note that $\overset{\leftarrow}{\epsrev} = \epsilon$
for all $n$ and $\epsilon$. In particular, if one fixes a matrix
$B_\epsilon \in LPM_n(\epsilon)$, by Proposition~\ref{Ptpm} the map
\begin{equation}\label{Esanity}
A \quad \mapsto \quad A^{-1} \quad \mapsto \quad P_n A^{-1} P_n = (P_n A
P_n)^{-1}
\end{equation}
is a diffeomorphism $: LPM_n(\epsilon) \to LPM_n(\epsrev)$ (where one
uses $P_n B_\epsilon^{-1} P_n$ in Cholesky-decomposing the cone
$LPM_n(\epsrev)$).
Similarly,~\eqref{Esanity} is a diffeomorphism $: TPM_n(\epsilon) \to
TPM_n(\epsrev)$ once one fixes $B_\epsilon \in TPM_n(\epsilon)$ -- again
for every $n \geq 1 $ and $\epsilon \in \{ \pm 1 \}^n$. This provides
another choice for $\epsilon' = \epsrev$, alongside~\eqref{Ecomposition}.
In particular,
\begin{equation}
\rev{\D_\epsilon} = P_n \D_\epsilon P_n = \D_{\epsrev}.
\end{equation}
\end{remark}

\begin{proof}[Proof of Proposition~\ref{Ptpm}]
Since $P_n A P_n$ reverses the rows and columns of a square matrix, it
interchanges the leading and trailing $k \times k$ principal minors for
every $k$. As linear bijections are indeed (smooth) diffeomorphisms, half
of the result is proved. Moreover, the maps in~\eqref{Ediffeos} commute.

For the inverse map, recall Jacobi's complementary minor
formula~\cite{Jacobi}: given integers $0 < p < n$, an invertible matrix
$A_{n \times n}$, and equi-sized subsets $J = \{ j_1 < \cdots < j_p \}, K
= \{ k_1 < \cdots < k_p \} \subseteq [n]$,
\begin{equation}\label{Ejacobi}
\det A \cdot \det (A^{-1})_{K^cJ^c} = (-1)^{j_1 + k_1 + \cdots + j_p +
k_p} \det A_{JK}.
\end{equation}
Here $J^c := [n] \setminus J$, and similarly for $K^c$. Now
apply~\eqref{Ejacobi} to $K = J = [k]$ for some $0<k<n$. If $A \in
LPM_n(\epsilon)$, then the trailing principal minors of $A^{-1}$ are:
\[
\det (A^{-1})_{[k]^c[k]^c} = \frac{\det A_{[k][k]}}{\det A},
\]
and this has sign $\epsilon_n \epsilon_k$ if $0 < k < n$. For $k=n$,
$\det A^{-1} = \frac{1}{\det A}$ has sign $\epsilon_n$. Moreover, this
entire process is reversible, so that one could have started with $A \in
TPM_n(\epsrev)$ and used $K = J = [k]^c$ instead. Finally, as $LPM_n \cup
TPM_n \subset GL_n(\R)$, the determinant map is nonzero-valued and hence
its reciprocal is smooth, whence so is the inverse map by Cramer's rule.
\end{proof}

Proposition~\ref{Ptpm} immediately implies

\begin{proof}[Proof of Theorem~\ref{Ttpm}]
This is now clear using either map from~\eqref{Ediffeos}. For
completeness, we record the calculations, fixing matrices $B_\epsilon \in
LPM_n(\epsilon)$ and $B_{\epsrev} \in LPM_n(\epsrev)$. Now given any $L
\in {\bf L}_n$,
\begin{align}
A = L B_\epsilon L^T &\ \in LPM_n(\epsilon) \ \implies \ 
P_n A P_n  = (P_n L P_n) \cdot P_n
B_\epsilon P_n \cdot (P_n L P_n)^T \in TPM_n(\epsilon),\\
A  = L B_{\epsrev} L^T &\ \in LPM_n(\epsrev) \ \implies \ A^{-1} =
L^{-T} \cdot B_{\epsrev}^{-1} \cdot L^{-1} \in TPM_n(\epsilon),
\end{align}
where the final step in each line follows from Proposition~\ref{Ptpm}:
$P_n B_\epsilon P_n$ and $B_{\epsrev}^{-1}$ lie in $TPM_n(\epsilon)$.
The final assertions involving the reversal map and the commuting diagram
are easily verified.
\end{proof}

\begin{remark}
In this proof and henceforth, we denote $L^{-T} := (L^{-1})^T$ for an
invertible matrix $L$.
\end{remark}

\section{Riemannian geometry of LPM matrices and group structure of each
$LPM_n(\epsilon)$}\label{Sriemannian}

Thus far, we have obtained a topological, even smooth structure of each
cone $LPM_n(\epsilon)$. We now study how our Cholesky decomposition on
$LPM_n(\epsilon)$ leads to a Riemannian-geometric and Lie-theoretic
structure.

We begin by addressing a natural question: akin to $PD_n$, every
$LPM_n(\epsilon)$ is a cone -- which is moreover open, in light of the
diffeomorphisms onto Cholesky space ${\bf L}_n$ or an open Euclidean
ball. But are these cones always convex? This turns out to be not always
true:

\begin{example}\label{Exnotconvex}
Given $\epsilon_1 \in \{ \pm 1 \}$, scalars $a,c \neq 0$ of sign
$\epsilon_1$, and $b > \sqrt{ac}$, the matrices $A_\pm := \begin{pmatrix}
a & \pm b \\ \pm b & c \end{pmatrix} \in LPM_2((\epsilon_1,-1))$.
However, $A_+ + A_-$ has a positive $2 \times 2$ determinant.

In particular, $LPM_n((\epsilon_1,-1,\epsilon_3,\dots,\epsilon_n))$ is
not convex either, for any $n \geq 2$ and signs $\epsilon_1$ and
$\epsilon_3, \dots, \epsilon_n \in \{ \pm 1 \}$. Indeed, one can simply
block-adjoin ${\rm diag}(-\epsilon_3, \epsilon_3 \epsilon_4, \dots,
\epsilon_{n-1} \epsilon_n)$ to $A_\pm$. \qed
\end{example}

Given this lack of convexity, not all $LPM_n(\epsilon)$ contain every
Euclidean geodesic, i.e.\ line segment. However, it turns out that they
all admit Riemannian manifold structures, which we now describe in
detail. The following treatment is along the lines of~\cite{Cholesky}.

\subsection{Riemannian metric on Cholesky space}

In~\cite[Section~3]{Cholesky}, Lin first develops a Riemannian structure
on Cholesky space ${\bf L}_n$, which he terms the \textit{log-Cholesky
metric}. Here is a summary.

Let $\LL_n$ denote the space of real lower triangular $n \times n$
matrices. (In Lin's notation, $\LL_n \leftrightarrow \mathcal{L}$ and
${\bf L}_n \leftrightarrow \mathcal{L}_+$.) Then $\LL_n \cong
\R^{n(n+1)/2}$ is a flat space and ${\bf L}_n$ is a convex subset that is
a smooth submanifold. The tangent space at every $L \in {\bf L}_n$ can be
identified with $\LL_n$. We now describe the Riemannian manifold
structure of ${\bf L}_n$.

\begin{theorem}[{\cite[Section~3.1]{Cholesky}}]\label{TLnMetric}\hfill
\begin{enumerate}[$(1)$]
\item The symmetric maps $\{ \widetilde{g}_L : T_L({\bf L}_n)^2 \to \R
\}_{L \in {\bf L}_n}$ given by
\begin{equation}
\widetilde{g}_L(X,Y) := \sum_{i>j} x_{ij} y_{ij} + \sum_{j=1}^n x_{jj}
y_{jj} l_{jj}^{-2}, \qquad X = (x_{ij}), Y = (y_{ij}) \in \LL_n
\end{equation}
provide a Riemannian metric on ${\bf L}_n$, where $l_{jj}$ are the
diagonal entries of $L$.

\item Given a square matrix $L \in \R^{n \times n}$, define its
\emph{diagonal part} and \emph{strictly lower triangular part},
respectively, as:
\begin{equation}
\D(L) := {\rm diag}(l_{11}, \dots, l_{nn}), \qquad \floo{L}_{ij} :=
\begin{cases}
l_{ij}, \quad &\text{if } i>j,\\
0, &\text{otherwise}.
\end{cases}
\end{equation}
Now the geodesic in ${\bf L}_n$ starting at $L$ in the direction $X \in
T_L {\bf L}_n$ is given by
\begin{equation}\label{Egeodesic}
\widetilde{\gamma}_{L,X}(t) := \floo{L} + t \floo{X} + \D(L) \exp( t
\D(X) \D(L)^{-1}), \qquad t \in \R.
\end{equation}
In particular, the geodesic in ${\bf L}_n$ from $L$ to $K$ in unit time is
\begin{equation}
\widetilde{\gamma}_{L,X}(t), \quad \text{with} \quad
X = \floo{K} - \floo{L} + \{ \log \D(K) - \log \D(L) \} \D(L).
\end{equation}

\item The {\em log-Cholesky metric} on ${\bf L}_n$ -- i.e.\ the induced
Riemannian distance between two points $L = (l_{ij})$ and $K = (k_{ij})$
-- is given by:
\begin{equation}\label{ElogCholesky}
d_{{\bf L}_n}(L,K) = \left( \sum_{i>j} (l_{ij} - k_{ij})^2 + \sum_{j=1}^n
(\log l_{jj} - \log k_{jj})^2 \right)^{1/2}.
\end{equation}
\end{enumerate}
\end{theorem}

\subsection{Riemannian metrics on LPM and TPM matrices}

We now transport the above Riemannian metric onto each cone
$LPM_n(\epsilon)$. By Theorem~\ref{Tlpm}, each of these cones is smoothly
diffeomorphic to $PD_n$, which is a smooth submanifold of the (flat)
vector space $\symm_n$ of real symmetric matrices. Moreover, it is easy
to check that one can perturb $A \in LPM_n(\epsilon)$ by $tX$, for any
real symmetric matrix $X$ and sufficiently small $|t|$, and yet stay in
$LPM_n(\epsilon)$. Thus, the tangent space is $T_A(LPM_n(\epsilon)) \cong
\symm_n$.

In the sequel, we will Cholesky-factorize matrices in $LPM_n(\epsilon)$
using not an arbitrary matrix $B_\epsilon$, but a specific choice: the
diagonal matrix $B_\epsilon = \D_\epsilon$ as in~\eqref{Elpmdiag}.
Fixing $\epsilon$ and $\D_\epsilon$, $\Phi_{\D_\epsilon}$ is a smooth
diffeomorphism by Theorem~\ref{Tlpm}. (In the special case $PD_n$ with
$\epsilon = {\bf 1}_n$, the maps $\Phi_{\D_\epsilon},
\Phi_{\D_\epsilon}^{-1}$ were respectively called $\mathscr{S},
\mathscr{L}$ by Lin in~\cite{Cholesky}.) Now we generalize the analysis
in~\cite{Cholesky} to all $\epsilon$:

\begin{theorem}\label{TLPMnMetric}
Fix $n \geq 1$, a sign pattern $\epsilon \in \{ \pm 1 \}^n$, and the
matrix $B_\epsilon = \D_\epsilon$. Denote $\Phi :=
\Phi_{\D_\epsilon}$.
\begin{enumerate}[$(1)$]
\item Given $A \in \R^{n \times n}$, define $\half{A} := \floo{A} +
\frac{1}{2} \D(A)$. Then for all $L \in {\bf L}_n$, the differential of
$\Phi$ at $L$ is
\begin{equation}\label{Edifferential}
D_L \Phi : T_L({\bf L}_n) \to T_{L \D_\epsilon L^T} (LPM_n(\epsilon)),
\quad \text{given by} \quad (D_L \Phi)(X) = L \D_\epsilon X^T + X
\D_\epsilon L^T.
\end{equation}
This is a linear isomorphism $: \LL_n \to \symm_n$ between the tangent
spaces, with inverse map given by:
\begin{equation}\label{Einverse}
(D_L \Phi)^{-1}(W) = L \half{(L^{-1} W L^{-T})} \D_\epsilon, \qquad
\forall W \in \symm_n.
\end{equation}

\item $LPM_n(\epsilon)$ is a Riemannian manifold under the {\em
log-Cholesky metric}:
\begin{equation}
g_{L \D_\epsilon L^T}(W,V) = \widetilde{g}_L \left( L \half{(L^{-1} W
L^{-T})} \D_\epsilon, L \half{(L^{-1} V L^{-T})} \D_\epsilon \right),
\quad \forall L \in {\bf L}_n, \ W,V \in \symm_n.
\end{equation}

\item The diffeomorphism $\Phi = \Phi_{\D_\epsilon} : ({\bf L}_n,
\widetilde{g}) \to (LPM_n(\epsilon), g)$ is in fact a Riemannian
isometry. Thus $D \Phi_{\D_\epsilon}$ transfers from ${\bf L}_n$ the
Riemannian metric, geodesic, and distance to $LPM_n(\epsilon)$:
\begin{align}
\widetilde{g}_L(X,Y) = &\ g_{\Phi(L)}((D_L \Phi)(X), (D_L \Phi)(Y)),\\
\gamma_{A,W}(t) = &\ \Phi(\widetilde{\gamma}_{\Phi^{-1}(A), X}(t)) =
\widetilde{\gamma}_{\Phi^{-1}(A), X}(t) \cdot \D_\epsilon \cdot
\widetilde{\gamma}_{\Phi^{-1}(A), X}(t)^T,\label{Egeodesic2}\\
d_{LPM_n(\epsilon)}(A,B) = &\ d_{{\bf L}_n}(\Phi^{-1}(A), \Phi^{-1}(B)),
\end{align}
where $X = L \half{(L^{-1} W L^{-T})} \D_\epsilon$ in~\eqref{Egeodesic2}.
\end{enumerate}
\end{theorem}

\begin{proof}
We begin by proving~(1). Let $\gamma : (-\delta, \delta) \to {\bf L}_n$
be any curve -- for some $\delta>0$ small enough -- with $\gamma(0) = L$
and $\gamma'(0) = X$. (For instance Lin uses $L + tX$ in~\cite{Cholesky};
one can also use the geodesic $\widetilde{\gamma}_{L,X}(t)$
from~\eqref{Egeodesic}.) Now $\Phi(\gamma(t)) = \gamma(t) \D_\epsilon
\gamma(t)^T$ is a curve in $LPM_n(\epsilon)$ through $L \D_\epsilon L^T$, so
\[
\left. \frac{d}{dt} \Phi(\gamma(t)) \right|_{t=0} = \gamma(0) \D_\epsilon
\gamma'(0)^T + \gamma'(0) \D_\epsilon \gamma(0)^T.
\]

This gives the linear map~\eqref{Edifferential} between the
equi-dimensional spaces $\LL_n$ and $\symm_n$. It suffices to show the
map is onto. Given $L \in {\bf L}_n$ and $W \in \symm_n$, we solve:
\begin{equation}\label{Esolving}
L \D_\epsilon X^T + X \D_\epsilon L^T = W \quad \implies \quad
L^{-1} W L^{-T} = \D_\epsilon (L^{-1}X)^T + (L^{-1}X) \D_\epsilon.
\end{equation}

As $L^{-1}X$ is lower triangular, $L^{-1} X \D_\epsilon = \half{(L^{-1} W
L^{-T})}$. This yields~\eqref{Einverse}, as $\D_\epsilon =
\D_\epsilon^{-1}$. Hence the linear map $D_L \Phi$ is onto (and the
uniqueness of the solution also shows $D_L \Phi$ is injective).

The above proves~(1). Given this, the diffeomorphism $\Phi$ induces a
Riemannian metric on $LPM_n(\epsilon)$, and~(2) follows. Moreover,
\cite[Definition~7.57]{Lee} shows that $\Phi = \Phi_{\D_\epsilon} : {\bf
L}_n \to LPM_n(\epsilon)$ is a Riemannian isometry. But then one deduces
the properties of $LPM_n(\epsilon)$ in~(3) from the corresponding
analogues over ${\bf L}_n$ -- and these were already shown in
Theorem~\ref{TLnMetric}.
\end{proof}

\begin{remark}\label{Runique}
When solving~\eqref{Esolving} above, it was convenient for Riemannian
geometry purposes that $\D_\epsilon$ be a diagonal matrix (in particular,
our Cholesky-type factorization is not the LDU decomposition). As
remarked early on in~\eqref{Elpmdiag}, there exists a unique such matrix
with unit-modulus entries in each cone $LPM_n(\epsilon),
TPM_n(\epsilon)$, though $\D_\epsilon$ here can be more general.
\end{remark}

\begin{remark}
Theorem~\ref{TLPMnMetric} has a natural counterpart for TPM matrices; we
do not state it here.
\end{remark}

\subsection{Abelian Lie group structures compatible with $\Phi$}

We now move towards the Euclidean space structure on each cone
$LPM_n(\epsilon)$. The metric is already introduced above; the next step
is to add in a group operation (the analogue of addition). This was
introduced in~\cite{Cholesky}:
\begin{equation}
\circledcirc : {\bf L}_n \times {\bf L}_n \to {\bf L}_n, \quad
\text{defined by} \quad L \circledcirc K := \floo{L} + \floo{K} + \D(L)
\D(K).
\end{equation}

We make two remarks here: first, the scalar multiplication and inner
product operations were not introduced in~\cite{Cholesky}, they are
introduced below. Second, it is somewhat misdirecting to work with
$\circledcirc$ on the larger space $\LL_n$ (as was done
in~\cite{Cholesky}), because while ${\bf L}_n$ is a subset of $\LL_n$, it
is not a subgroup or a submanifold. Thus, we do not work with
$\circledcirc$ on $\LL_n$.

Lin then showed:

\begin{theorem}[{\cite[Section~3.3]{Cholesky}}]\hfill
\begin{enumerate}[$(1)$]
\item The operation $\circledcirc$ and the inverse map
$L_\circledcirc^{-1} := -\floo{L} + \D(L)^{-1}$ are both smooth, and make
${\bf L}_n$ into an abelian group with identity $\Id_n$.

\item Given $K \in {\bf L}_n$, the differential at $L \in {\bf L}_n$ of
the left (= right) translation $\ell_K$ is:
\begin{equation}
D_L \ell_K : X \mapsto \floo{X} + \D(K) \D(X), \qquad X \in \LL_n,
\end{equation}
and the Riemannian metric is compatible with these:
\begin{equation}
\widetilde{g}_{K \circledcirc L}((D_L \ell_K)(X), (D_L \ell_K)(Y)) =
\widetilde{g}_L(X,Y) \qquad \forall K,L \in {\bf L}_n, \ X,Y \in \LL_n.
\end{equation}
Thus 
$({\bf L}_n, \circledcirc, \Id_n, L_\circledcirc^{-1}; \widetilde{g})$ is
an abelian Lie group, and $\widetilde{g}$ is a bi-invariant Riemannian
metric.
\end{enumerate}
\end{theorem}

We now carry out the same thing in each LPM cone, extending the special
case of $\epsilon = {\bf 1}_n$ by Lin in~\cite{Cholesky}. Since
$\Phi_{\D_\epsilon}^{\pm 1}$ are smooth Riemannian isometries from above,
we use these to transfer the group operation smoothly to
$LPM_n(\epsilon)$:
\begin{equation}
A \circledast B := \Phi_{\D_\epsilon}( \Phi_{\D_\epsilon}^{-1}(A)
\circledcirc \Phi_{\D_\epsilon}^{-1}(B) ) = (\Phi_{\D_\epsilon}^{-1}(A)
\circledcirc \Phi_{\D_\epsilon}^{-1}(B)) \cdot \D_\epsilon \cdot
(\Phi_{\D_\epsilon}^{-1}(A) \circledcirc \Phi_{\D_\epsilon}^{-1}(B))^T.
\end{equation}

Thus, $\Phi_{\D_\epsilon}^{\pm 1}$ are group isomorphisms by definition.
Combined with Theorem~\ref{TLPMnMetric}, this yields:

\begin{theorem}
For all $n \geq 1$ and $\epsilon \in \{ \pm 1 \}^n$, the space
$(LPM_n(\epsilon), \circledast)$ is an abelian Lie group, with identity
element $\D_\epsilon$ and inverse map
\begin{equation}
A \mapsto (\Phi_{\D_\epsilon}^{-1}(A))_\circledcirc^{-1}
\cdot \D_\epsilon \cdot (\Phi_{\D_\epsilon}^{-1}(A))_\circledcirc^{-T}.
\end{equation}
Moreover, the Riemannian metric $g$ is bi-invariant with respect to this
(smooth) group action.
\end{theorem}

Indeed, the smoothness, commutativity, and translation-invariance of
$\circledast$ follow from those of $\circledcirc$, since
$\Phi_{\D_\epsilon}^{\pm 1}$ are smooth group isomorphisms. Moreover, the
treatment in this section can be suitably repeated for each
$TPM_n(\epsilon)$ -- or apply the linear map $A \mapsto P_n A P_n$ over
$LPM_n(\epsilon)$. We thus have:

\begin{proof}[Proof of Theorem~\ref{Triemannian}]
All but the zero sectional curvature part for $LPM_n(\epsilon)$ follow
from the above analysis.
The zero sectional curvature follows either from (the proof of)
\cite[Proposition~8]{Cholesky}, or from Theorem~\ref{Teuclidean} which we
prove presently.
Now applying the linear isometric isomorphism of abelian metric groups $A
\mapsto P_n A P_n$, the results follow for all $TPM_n(\epsilon)$ too.
\end{proof}

Along related lines is the following observation; the proof is direct.

\begin{lemma}\label{Lklein4}
The self-maps of ${\bf L}_n$ given by
\[
L \mapsto L, \quad
L \mapsto L_\circledcirc^{-1}, \quad
L \mapsto (P_n L P_n)^T, \quad
L \mapsto (P_n L_\circledcirc^{-1} P_n)^T
\]
form a commuting Klein-4 subgroup of maps, each of which is an isometric
automorphism of the abelian Lie group ${\bf L}_n$. (Hence this structure
transfers via $\Phi$ to the Lie groups $LPM_n(\epsilon),
TPM_n(\epsilon)$.)
\end{lemma}

The third of these maps comes from~\eqref{Ediffeos}; but note that $L
\mapsto L^{-1}$ is neither a group-map nor an isometry. For instance, for
$L = \begin{pmatrix} 1 & 0 \\ 2 & 2 \end{pmatrix}$, we have
$L^{-1} \ \begin{pmatrix} 1 & 0 \\ -1 & 1/2 \end{pmatrix}$ and
$(L_{\circledcirc}^{-1})^{-1} = \begin{pmatrix} 1 & 0 \\ 4 & 2
\end{pmatrix}$.
So
\[
d(L, \Id_2) = \sqrt{4+(\log 2)^2} \quad \neq \quad
d(L^{-1}, \Id_2) = \sqrt{1+(\log 2)^2}.
\]
Similarly, $L^{-1} \circledcirc (L_{\circledcirc}^{-1})^{-1} =
\begin{pmatrix} 1 & 0 \\ 3 & 1 \end{pmatrix} \neq \Id_2
= (L \circledcirc L_{\circledcirc}^{-1})^{-1}$.
Also, $(L^{-1})_\circledcirc^{-1} = \begin{pmatrix} 1 & 0 \\ 1 & 1/2
\end{pmatrix} \neq (L_{\circledcirc}^{-1})^{-1}$.

\begin{remark}
Another observation is: extend $\circledcirc$ to all square matrices by
defining $\ceil{A}$ to be the strictly upper triangular part of $A$, such
that $A = \floo{A} + \D(A) + \ceil{A}$; and now define
\begin{equation}\label{Eextend}
A \circledcirc A' := \floo{A} + \floo{A'} + \ceil{A} + \ceil{A'} + \D(A)
\D(A').
\end{equation}
Now if $A = L \D_\epsilon L^T, A' = L' \D_\epsilon (L')^T \in
LPM_n(\epsilon)$, then $A \circledast A' = (L \circledcirc L')
\D_\epsilon (L^* \circledcirc (L')^*)$.
\end{remark}

\begin{remark}[Riemannian exponentials and logarithms; parallel
transport]
The above treatment can be carried out (with minimal modifications) for
the cones $TPM_n(\epsilon)$ as well. Moreover, in~\cite{Cholesky}, the
Riemannian exponential and logarithm maps over ${\bf L}_n$ were also
written down, and one can then transfer them to each $LPM_n(\epsilon)$
and $TPM_n(\epsilon)$. Similarly, parallel transport along geodesics in
${\bf L}_n$, developed in~\cite{Cholesky}, can be transported in parallel
to all cones $LPM_n(\epsilon)$ and $TPM_n(\epsilon)$. We will not
elaborate further on the details of these ``verbatim'' geometric
constructions to~\cite{Cholesky}.
\end{remark}

\section{Cholesky (complex) manifolds are Euclidean spaces, whose
union yields Hilbert space}
\label{Seuclidean}

This section goes beyond~\cite{Cholesky} and first extends the abelian
Lie group and Riemannian manifold structure of Cholesky space from
\cite{Cholesky} (recorded in Section~\ref{Sriemannian}), to that of a
finite-dimensional $\mathbb{R}$-vector space -- and even beyond, to an
inner product space.

\begin{theorem}\label{Tinnerproduct}
Fix $n \geq 1$.
\begin{enumerate}[$(1)$]
\item Define ``scalar multiplication''
$\cdot : \R \times {\bf L}_n \to {\bf L}_n$ via:
$\alpha \cdot L := \alpha \floo{L} + \D(L)^\alpha$.
Then the map
\begin{equation}
\eta : ({\bf L}_n, \circledcirc, \cdot) \to (\R^{n(n+1)/2}, +, \cdot)
\quad \text{defined by} \quad \eta(L) := (\log l_{11}, \dots, \log
l_{nn} \, ; \, \{ l_{ij} : i > j \})
\end{equation}
is an $\R$-vector space isomorphism.

\item Also define the form $\tangle{\cdot,\cdot} : {\bf L}_n^2 \to {\bf
L}_n$ defined by
\begin{equation}
\tangle{L,K} := \sum_{i>j} l_{ij} k_{ij} + \sum_{j=1}^n \log(l_{jj})
\log(k_{jj}).
\end{equation}
Then $\tangle{\cdot,\cdot}$ is $\R$-bilinear, and for all $L,K \in {\bf
L}_n$ we have
\[
d_{{\bf L}_n}(L,K)^2 = \tangle{L-'K,L-'K}, \qquad \text{where} \qquad
L-'K := L \circledcirc K_{\circledcirc}^{-1} = \eta^{-1}(\eta(L) -
\eta(K)).
\]
Therefore $({\bf L}_n, \tangle{\cdot,\cdot})$ is isometrically isomorphic
to the Euclidean space $\R^{n(n+1)/2}$.

\item These operations commute with the reversal map~\eqref{Ereversal} on
Cholesky space -- i.e., $L \mapsto \rev{L}$ is an isometric isomorphism
of Hilbert spaces. Moreover, $\floo{\rev{L}} =
\overset{\longleftarrow}{\floo{L}}$ and $\D(\rev{L}) =
\overset{\longleftarrow}{\D(L)}$ on ${\bf L}_n$.
\end{enumerate}
\end{theorem}

The proofs are straightforward. Moreover, Theorem~\ref{Tinnerproduct}
immediately implies Theorem~\ref{Teuclidean}.

\begin{remark}\label{RVectorSpace}
Theorems~\ref{Tinnerproduct} and~\ref{Teuclidean} resemble the main
conclusions of~\cite{VectorSpace} -- which showed an analogous result for
an alternate structure on the cone $PD_n$. Namely, the exponential map
\begin{equation}
\exp : (\symm_n,+) \to (PD_n, \odot), \qquad \exp(A) \odot
\exp(B) := \exp(A+B)
\end{equation}
was shown to be an isometric isomorphism which makes $PD_n$ into an
abelian Lie group, with translation-invariant metric $d(\exp(A), \exp(B))
= \| \log( \exp(-A/2) \exp(B) \exp(-A/2) ) \|$. Moreover, defining scalar
multiplication as $\alpha \cdot' \exp(A) := \exp(\alpha A)$ shows that
$\exp : \symm_n \to PD_n$ is in fact a $\mathbb{R}$-vector space
isomorphism. In particular, $PD_n$ has sectional curvature zero for this
metric.
\end{remark}

As a consequence of Theorems~\ref{Tinnerproduct}(3) and~\ref{Ttpm}, we
have:

\begin{cor}\label{Cadditive}
For all $n \geq 1$ and $\epsilon \in \{ \pm 1 \}^n$, the reversal map $:
LPM_n(\epsilon) \to TPM_n(\epsilon)$ is an isometric isomorphism of
Hilbert spaces.
\end{cor}

\begin{proof}
Note that the Hilbert space structures on $LPM_n(\epsilon),
TPM_n(\epsilon)$ are simply transfers from Cholesky space ${\bf L}_n$ via
the maps $\Phi_{\D_\epsilon}, \Phi^{\D_{\epsrev}}$ respectively. Now the
result is easily verified, using~\eqref{Eextend} to define $\circledcirc$
on upper triangular matrices.
\end{proof}

We now go even further beyond~\cite{Cholesky} to infinite-dimensional
inner product spaces, followed by the Hermitian analogues of them and of
$LPM_n(\epsilon), TPM_n(\epsilon)$.

\begin{proof}[Proof of Theorem~\ref{Thilbert}]
As we will work here and in proving Theorem~\ref{Thermitian} with the
same structures over real and complex spaces, big and small, we first
write down the proposed \textit{real} inner product space operations over
the largest cone ${\bf L} = {\bf L}^\C_{\mathcal{H}} \cong ({\bf
L}^\R_{\mathcal{H}})^2$, and can then restrict to smaller spaces:
\begin{align}
L \circledcirc K = &\ \floo{L} + \floo{K} + \D(L) \D(K),\\
L_\circledcirc^{-1} = &\ -\floo{L} + \D(L)^{-1},\\
\alpha \cdot L = &\ \alpha \floo{L} + \D(L)^\alpha \quad \forall \alpha
\in \R,\\
\tangle{L,K} = &\ \sum_{i>j\geq 1} \overline{l_{ij}} k_{ij} + \sum_{j
\geq 1} \log(l_{jj}) \log(k_{jj}),\\
d_{\bf L}(L,K) = &\ \tangle{L-'K,L-'K} := \tangle{ L \circledcirc
K_\circledcirc^{-1}, L \circledcirc K_\circledcirc^{-1}},\\
\eta(L) = &\ (\log l_{11};\ l_{21}, \log l_{22};\ \dots;\ l_{n1}, \dots,
l_{n,n-1}, \log l_{nn};\ \dots).
\end{align}

One checks that these operations restrict to ${\bf L}_{\mathcal{H}} =
{\bf L}^\R_{\mathcal{H}}$ and further to ${\bf L}_{00} = {\bf L}_{00}^\R$
-- as well as that $\eta$ is an isometric $\R$-vector space isomorphism
on these cones.
Straightforward verifications now show the real inner product structure
on the (possibly complex) ${\bf L}$-cones. E.g.\ $\eta$ is compatible
with the inclusions in~\eqref{Esquare}, so it extends from the cones
${\bf L}_n$ to their direct limit ${\bf L}_{00}$. Thus ${\bf L}_{00}
\cong c_{00}$; as these are respectively dense in ${\bf L}_{\mathcal{H}}$
and $\ell^2$, and the isometry $\eta$ takes Cauchy sequences to Cauchy
sequences, it provides a unique extension to the closure that is
consistent with its definition there.

Having shown the results for the ${\bf L}$-spaces, the results for the
$LPM(\beps)$-cones follow as the structures transfer via $\Phi$. For
instance, the commuting squares~\eqref{Esquare} imply
$\Phi_{\D_{\beps}}({\bf L}_{00}) = LPM_{00}(\beps)$. As the isometry
$\Phi_{\D_{\beps}}$ is Cauchy-continuous, it extends to an isomorphism of
the closures.
\end{proof}

We next introduce:

\begin{defn}\label{Dlpm2}
Given a sign sequence $\beps \in \{ \pm 1 \}^\infty$, define $LPM(\beps)$
to be the set of real symmetric semi-infinite matrices whose leading $k
\times k$ principal minor has sign $\epsilon_k$, for all $k \geq 1$.
Inside this, define $LPM_{\mathcal{H}}(\beps)$ to be the image under
$\Phi_{\D_{\beps}}$ of ${\bf L}_{\mathcal{H}}$. Similarly define
$LPM^\C(\beps) \supset LPM^\C_{\mathcal{H}}(\beps)$.
\end{defn}

Now just as $LPM_{00}(\beps) = \Phi_{\D_{\beps}}({\bf L}_{00})$, it is
natural to ask what is the image $LPM_{\mathcal{H}}(\beps) =
\Phi_{\D_{\beps}}({\bf L}_{\mathcal{H}})$. One can check via translating
in the language of $\Phi_{\D_{\beps}}$, this is precisely
\begin{equation}\label{Eset}
\left\{ A \in LPM(\beps) \, : \, \sum_{j \geq 1} \left( \log
\frac{\epsilon_{j+1} \det A_{[j+1][j+1]}}{\epsilon_j \det A_{[j][j]}}
\right)^2 < \infty,\ \sum_{j \geq 1} \epsilon_j \epsilon_{j+1} \left(
a_{jj} - \frac{\det A_{[j+1][j+1]}}{\det A_{[j][j]}} \right) < \infty
\right\}.
\end{equation}

\noindent We now show the Cholesky decomposition and real Hilbert space
structure for Hermitian matrices.

\begin{proof}[Proof of Theorem~\ref{Thermitian}]
The first part is proved in the same way as Theorem~\ref{Tlpm}. Ditto for
the second part, modulo the obvious changes, e.g.\ $l(mp+uq) \leadsto l(m
\overline{p} + \overline{u} q)$, or $L^T \leadsto L^*$, or ${\bf b}^T
\leadsto {\bf b}^*$. The solution now uses $ac - |b|^2$, to yield
$\displaystyle q = \sqrt{\frac{ac-|b|^2}{a} \cdot \frac{m}{mv-|u|^2}}$.
We add here that the real Euclidean ball to which all
$LPM^\C_n(\epsilon)$, $TPM^\C_n(\epsilon)$, and ${\bf L}^\C_n$ are
diffeomorphic, is $B_{\R^{n^2}}({\bf 0},1)$.

The third part is proved in the same way as Theorem~\ref{Ttpm}. For the
fourth part, we use the maps defined above, as well as the \textit{real}
vector space map:
\begin{equation}
\eta^\C(L) := (\log l_{11}; \ \Re(l_{21}), \Im(l_{21}), \log l_{22}; \
\Re(l_{31}), \Im(l_{31}), \Re(l_{32}), \Im(l_{32}), \log l_{33}; \
\dots).
\end{equation}
As this is again an isometric $\R$-vector space isomorphism, the
other verifications are as over $\R$.
\end{proof}

\begin{remark}[Fr\'echet means and barycentres -- in Hilbert space]
In~\cite{Cholesky}, the author used their machinery to discuss Fr\'echet
means of random matrices in ${\bf L}_n$, and log-Cholesky means for e.g.\
$PD_n$-valued random matrices $A$ with finite second moment:
$\mathbb{E}[ d_{PD_n}(A,A_0)^2 ] < \infty$. For instance, the barycentre
of finitely many matrices in either space.
In~\cite[Proposition~9]{Cholesky}, the existence and uniqueness of this
was proved via using that ${\bf L}_n$ and $\LL_n$ are homeomorphic, so
${\bf L}_n$ is simply connected and has zero sectional curvature, and
hence one can apply results of Bhattacharya and Patrangenaru~\cite{BP}
for complete, simply connected Riemannian manifolds.
However, now it is clear that one does not need~\cite{BP}, since the
homeomorphism $: {\bf L}_n \to \LL_n \cong \R^{n(n+1)/2}$ is in fact an
isometric isomorphism of Hilbert spaces (Theorem~\ref{Teuclidean}). Hence
``usual'' multivariate analysis and finite-dimensional probability
applies to it.

In fact, one can do more: Fr\'echet means/barycentres can also be
computed in the larger (real Hilbert) spaces $LPM_{\mathcal{H}}^\R({\bf
1}_\infty), LPM_{\mathcal{H}}^\C({\bf 1}_\infty)$. Via the
diffeomorphisms $\Phi_{\D_\epsilon}$, $A \mapsto A^{-1}$, these
computations also carry over to every $LPM_n^\F(\epsilon),
LPM_{\mathcal{H}}^\F(\beps)$, and $TPM^\F_n(\epsrev)$ -- for $\F = \R$
and $\F = \C$.
\end{remark}

\subsection{Towers of abelian metric groups, including over other fields}

We have seen real inner product space structures on
$LPM_n^\C((\epsilon_1, \dots, \epsilon_n)), LPM_{00}^\C(\beps),
LPM_{\mathcal{H}}^\C(\beps)$ for all $\beps \in \{ \pm 1 \}^\infty$ and
$n \geq 1$. These spaces formed towers of inclusions, which we now
construct over other subfields $\F \subseteq \C$. 

\begin{defn}\label{Dsubfield}
Let $\mathbb{F}$ be any subfield of $\C$.
\begin{enumerate}[(1)]
\item Given $n \geq 1$ and $\epsilon \in \{ \pm 1 \}^n$, define
$LPM_n^\F(\epsilon)$ to be the Hermitian matrices in $\F^{n \times n}$
whose leading principal $k \times k$ minor has sign $\epsilon_k$, for all
$k \geq 1$.

\item Define the direct limit $LPM_{00}^\F(\beps)$ as in
Theorem~\ref{Teuclidean}(2), but with all entries in $\F$. (We do not
work with the TPM counterpart.) Also define $LPM^\F_{\mathcal{H}}(\beps)$
as in~\eqref{Eset}.

\item Define the ``completion'' ${\bf L}_{\mathcal{H}}^\F(\beps)$ as in
Theorem~\ref{Thilbert}(3), with all entries not real but in $\F$.

\item With the subscript-symbol $\star \in \{ 00, \mathcal{H} \}$, define
the unions
\begin{equation}
LPM^\F_n := \bigsqcup_{\epsilon \in \{ \pm 1 \}^n} LPM^\F_n(\epsilon),
\qquad LPM^\F_\star := \bigsqcup_{\beps \in \{ \pm 1 \}^\infty}
LPM^\F_\star(\beps),
\end{equation}

\item For these three choices, define the corresponding sets ${\bf
L}^\F_\star$ of lower-triangular matrices with entries in $\F$ and
diagonal entries in $\F \cap (0,\infty)$.

\item In all cases, given a (possibly semi-infinite) real symmetric
matrix $B_{\beps} \in LPM_\star^\F(\beps)$, define $\Phi_{B_{\beps}} :
{\bf L}_\star^\F \to LPM_\star^\F(\beps)$ via:  $\Phi_{B_{\beps}}(L) := L
B_{\beps} L^*$.
\end{enumerate}
\end{defn}

Then the inner product space structure recorded in
Theorems~\ref{Teuclidean}, \ref{Thilbert}, and~\ref{Thermitian}
also occurs here:

\begin{theorem}\label{Ttower}
Let $\F$ be any subfield of $\C$ with $\F \cap (0,\infty)$ closed under
positive square roots.\footnote{Examples include
(i)~the constructible numbers, important in Euclidean geometry, and
(ii)~the (real) algebraic numbers ($\overline{\mathbb{Q}} \cap
\mathbb{R}$ or) $\overline{\mathbb{Q}}$.
More generally, one can start with an arbitrary subfield $\mathbb{E}$ of
$\R$, and then
(i$'$)~inductively adjoin at each stage, the square roots of elements in
$(0,\infty)$ already obtained. This yields $\mathbb{F}$ with the above
property.
Alternately, one can (ii$'$)~take the algebraic closure of any subfield
$\mathbb{E} \subseteq \mathbb{C}$, and then use $\mathbb{F} =
\overline{\mathbb{E}}$ or $\overline{\mathbb{E}} \cap \R$.}

\begin{enumerate}[$(1)$]
\item Given a subscript-symbol $\star \in \{ n, 00,
\mathcal{H} \}$, and a sign tuple/sequence $\beps$, one has a tower of
commuting diagrams of monomorphic isometries of abelian groups with
bi-invariant metrics:
\begin{equation}\label{Etower}
   \begin{CD}
     {\bf L}^\F_1  @>>> {\bf L}^\F_2 @>>> \, \cdots \, @>>> {\bf
     L}^\F_{00} @>>> ({\bf L}^\F_{\mathcal{H}}, \circledcirc)\\
     @V\Phi_{\D_\epsilon}VV @V\Phi_{\D_\epsilon}VV @.
     @V\Phi_{\D_{\beps}}VV
     @V\Phi_{\D_{\beps}}VV\\
     LPM^\F_1((\epsilon_1))   @>>> LPM^\F_2 ((\epsilon_1, \epsilon_2))
     @>>> \, \cdots \, @>>> LPM^\F_{00}(\beps) @>>>
     (LPM^\F_{\mathcal{H}}(\beps), \circledast)
   \end{CD}
\end{equation}
Each of these is an ``additive'' 2-divisible subgroup of the Hilbert
space $({\bf L}_{\mathcal{H}}^\C, \tangle{\cdot,\cdot})$, and it has a
translation-invariant metric induced by the restriction of the
(bi-additive positive definite) inner product.

\item Each downward arrow $\Phi_{\D_\epsilon}$ or $\Phi_{\D_{\beps}}$,
from ${\bf L}^\F_\star$ to $LPM^\F_\star(\beps)$, is a homeomorphism that
is the restriction of a smooth diffeomorphism.

\item Using Theorem~\ref{Ttpm}, we can add another row of homeomorphisms
(which are also group morphisms) involving the cones
$TPM_n^\F(\epsilon)$, for all $n, \epsilon$.

\item If $\F$ is algebraically closed $\overline{\mathbb{E}}$, or equals
$\overline{\mathbb{E}} \cap \R$, then each of these 2-divisible groups is
in fact a $\mathbb{Q}$-vector subspace of ${\bf L}_{\mathcal{H}}$. More
strongly, this holds if $\F \cap (0,\infty)$ is closed under taking $q$th
roots for all primes $q \geq 2$.
\end{enumerate}
\end{theorem}

\begin{proof}
In the first part, to show that each of the maps $\Phi_{\D_{\beps}}$ is
one-to-one, say $L \D_{\beps} L^* = K \D_{\beps} K^*$. Then $(K^{-1} L)
\D_{\beps} = \D_{\beps} (L^{-1}K)^*$, so both sides are diagonal. This
implies $K^{-1} L$ is diagonal, with diagonal entries of
modulus $1$ and in $(0,\infty)$. So $K^{-1} L = \Id_{\star'}$, with
$\star' = n$ if $\star = n$, else $\star' = \infty$.

In the remaining parts, the closure of $\F \cap (0,\infty)$ under
positive square roots is needed in order to Cholesky-factorize matrices
in $LPM^\F_\star(\epsilon)$, by the algorithmic proof of
Theorem~\ref{Tlpm}.

This explains the hypothesis; we then proceed. We will only discuss the
2-divisibility and the final part. First, $\frac{1}{2} \cdot L =
\frac{1}{2} \floo{L} + \D(L)^{1/2}$ inside ${\bf L}_{\mathcal{H}}^\C$.
Now by the closure of $\F \cap (0,\infty)$ under $\sqrt{\cdot}$, it
follows that $\frac{1}{2} \cdot L \in {\bf L}^\F_\star$ for $L \in {\bf
L}^\F_\star$.
Similarly, if the hypothesis in the final part holds, and $p,q$ are
nonzero integers, then $\frac{p}{q} \cdot L = \frac{p}{q} \floo{L} +
\D(L)^{p/q}$, and this has all entries in $\F$.
\end{proof}

Akin to Lemma~\ref{Lklein4}, in the complex case we again have a
commuting -- and larger -- group of self-maps of complex Cholesky space.
More generally:

\begin{lemma}\label{Lklein8}
Let $\F$ be as Theorem~\ref{Ttower}, with $\F \nsubseteq \R$. Then the
self-maps of ${\bf L}_n^\F$ given by
\[
L \mapsto L_\circledcirc^{-1}, \quad L \mapsto (P_n L P_n)^*, \quad L
\mapsto \overline{L}
\]
pairwise commute, and generate a Boolean group $(\mathbb{Z} / 2
\mathbb{Z})^3$ of isometric automorphisms of ${\bf L}_n^\F$.
These maps are moreover $\mathbb{Q}$-linear if $\F$ is as in
Theorem~\ref{Ttower}(4), and $\R$-linear if $\F = \C$.
They transfer via $\Phi_{\D_\epsilon}^\F$ to $LPM_n^\F(\epsilon)$ for
each $\epsilon \in \{ \pm 1 \}^n$, and similarly to each
$TPM_n^\F(\epsilon)$.
\end{lemma}

\section{Alternate Lie group structure on the $LPM_n$ and $TPM_n$ cones;
probability}\label{Sbiggroup}

We now switch tracks from Cholesky decompositions and Riemannian
geometry, to exploring towers of groups of LPM matrices with a motivation
from probability theory.
Theorem~\ref{Ttower} revealed two towers of isomorphic subgroup-pairs
that were not just Riemannian manifolds but in fact additive subgroups of
$\R$-vector spaces with real inner products.

We now explore additional groups found in LPM spaces, which necessarily
cannot embed into any Banach space, yet possess translation-invariant
metrics. We first introduce the following notation.

\begin{defn}\label{Dschurprod}
Define the Schur product ${\bf a} \circ {\bf b}$ of two real tuples
${\bf a}, {\bf b}$ of equal length, to be the tuple of their
coordinatewise products.
\end{defn}

Now we have:

\begin{theorem}\label{TbiggroupsF}
Let $\F$ be any subfield of $\C$ with $\F \cap (0,\infty)$ closed under
positive square roots, and $\star$ a subscript-symbol in $\{ n, 00,
\mathcal{H} \}$. Define the binary operation $\boxdot$ on $LPM_\star^\F$
as follows: given matrices $A = L \D_{\beps} L^T$ and $A' = L'
\D_{\beps'} (L')^T$, with $L, L' \in {\bf L}^\F_\star$ (where
$\beps$ denotes $\epsilon \in \{ \pm 1 \}^n$ if $\star = n$), we set:
\begin{align}
A \boxdot A' := &\ (L \circledcirc L') \cdot (\D_{\beps} \D_{\beps'})
\cdot (L \circledcirc L')^T \in LPM_\star^\F(\beps \circ \beps'),\\
(L \D_{\beps} L^T)_\boxdot^{-1} := &\ L_\circledcirc^{-1} \D_{\beps}
L_\circledcirc^{-T}.
\end{align}
\begin{enumerate}[$(1)$]
\item This structure (and identity $\Id_{\star'}$) makes $LPM_\star^\F$
into an abelian group that is isomorphic to the direct product
$PD^\F_\star \times S_2^{\star'}$, where $\star' = n$ if $\star = n$, and
$\star' = \infty$ otherwise. On the subgroup $PD^\F_\star =
LPM^\F_\star({\bf 1}_{\star'})$, we have $\boxdot \equiv \circledast$.

\item This group is equipped with a family of bi-invariant
``$L^p$-norms'' that extend the metric on the subgroup $PD_n^\F$ and the
discrete metric (Kronecker delta) on $S_2^{\star'}$:
\begin{equation}\label{Edp}
d_p(L \D_{\beps} L^T, K \D_{\beps'} K^T) := \left\| (d_{{\bf
L}_\star}(L,K), 1 - \delta_{\beps,\beps'}) \right\|_p, \qquad p \in
[1,\infty].
\end{equation}

\item If $\star = \star' = n$, then one can similarly equip $TPM_n^\F$
with a parallel structure. If moreover $\F = \R$ or $\C$, the cones
$LPM_n^\F$ and $TPM_n^\F$ are abelian complete separable metric Lie
groups and Riemannian manifolds with bi-invariant metric (and sectional
curvature zero). Moreover, the reversal map is an isometric isomorphism
of Lie groups.

\item The nontrivial torsion elements of $LPM_\star^\F$ are precisely the
matrices $\D_{\beps}$ for all $\beps \in \{ \pm 1 \}^{\star'} \setminus
\{ {\bf 1}_{\star'} \}$.
\end{enumerate}
\end{theorem}

Before showing this result, we make two remarks. First, this result
clearly implies/subsumes Theorem~\ref{Tbiggroup}. Second, there are other
subgroups that also share these properties: let $G$ be any subgroup of
the torsion group $S_2^{\star'}$. Then $PD_\star^\F \times G$ is a
subgroup of $LPM_\star^\F$. Thus one can write down another tower of
subgroups for $G$ (under $\boxdot$) akin to~\eqref{Etower}.

\begin{proof}[Proof of Theorem~\ref{TbiggroupsF}]\hfill
\begin{enumerate}[(1)]
\item Consider the map $\varphi : PD^\F_\star \times S_2^{\star'} \to
LPM_\star^\F$, given by $(L L^*,\beps) \mapsto \Phi_{\D_{\beps}}(L) = L
\D_{\beps} L^*$ (where $L^*$ denotes conjugate-transpose, not $\star$).
By ``countably applying'' Theorem~\ref{Tlpm} (or its proof over $\F$,
suitably adapted to the complex case), $\varphi$ is a group map and a
bijection, which proves the first claim. The restriction to $PD_\star^\F$
is clear.

\item We compute:
\begin{align*}
&\ d_p((J \D_{\beps''} J^*) \boxdot (L \D_{\beps} L^*),
(J \D_{\beps''} J^*) \boxdot (K \D_{\beps'} K^*))\\
= &\ d_p(\varphi((J \circledcirc L)(J \circledcirc L)^*, \beps'' \circ
\beps), \varphi((J \circledcirc K)(J \circledcirc K)^*, \beps'' \circ
\beps'))\\
= &\ \| (d_{{\bf L}_\star}(J \circledcirc L, J \circledcirc K),
1 - \delta_{\beps'' \circ \beps, \beps'' \circ \beps'}) \|_p\\
= &\ \| (d_{{\bf L}_\star}(L,K), 1 - \delta_{\beps, \beps'}) \|_p
\end{align*}
since $d_{{\bf L}_\star}$ and the discrete metric on $S_2^{\star'}$ are
each translation-invariant. This shows the result.

\item The first claim is now standard, via~\eqref{Ediffeos}. Now let $\F
= \R$ or $\C$; then $LPM_n^\F = \bigsqcup_{\epsilon \in \{ \pm 1 \}^n}
LPM_n^\F(\epsilon)$ is a disjoint union of $2^n$-many cosets, each
isomorphic to the complete separable metric Lie group $PD_n^\F$ (which
has zero sectional curvature and a bi-invariant Riemannian metric, from
results proved above). The final assertion about the reversal map is also
easily verified, using that
(a)~this map is anti-multiplicative;
(b)~it commutes with $\circledcirc$ (including via~\eqref{Eextend}),
conjugate, and transpose; and
(c)~the metrics on these cones are defined via $\Phi_{\D_\epsilon}$ and
$\Phi^{\D_{\epsrev}}$, and these are intertwined by the reversal map via
Theorem~\ref{Ttpm}.

\item This follows from~(1), as $PD_\star^\F \subseteq PD_\star^\C$
(which is a Euclidean space), hence is torsionfree.
\qedhere
\end{enumerate}
\end{proof}

Before moving on, we mention two ``bigger'' groups not considered so far:

\begin{prop}\label{Plpm}
Given a field $\F \subseteq \C$ as in Theorem~\ref{TbiggroupsF}, and a
sign pattern $\beps \in \{ \pm 1 \}^\infty$, let $LPM^\F(\beps)$ consist
of all semi-infinite Hermitian matrices with entries in $\F$ and $k
\times k$ leading principal minor of sign $\epsilon_k$. Define their
union $LPM^\F := \bigsqcup_{\beps \in \{ \pm 1 \}^\infty} LPM^\F(\beps)$.
\begin{enumerate}[$(1)$]
\item Let ${\bf L}^\F$ denote the semi-infinite lower triangular matrices
with entries in $\F$ and diagonal entries in $\F\cap(0,\infty)$.
Given $B_{\beps} \in LPM^\F(\beps)$, the Cholesky map
$\Phi_{B_{\beps}} : {\bf L}^\F \to LPM^\F(\beps)$ is a bijection.

\item Both $({\bf L}^\F, \circledcirc)$ and hence $(LPM^\F(\beps),
\circledast)$ are again isomorphic abelian groups.

\item $(LPM^\F, \boxdot)$ is also an abelian group, with nontrivial
torsion elements.
\end{enumerate}
\end{prop}

We only remark here that the algorithmic proof of Theorem~\ref{Tlpm}
extends to the present case, when carried out inductively for all $n$.

Notice a difference between Proposition~\ref{Plpm} and
Theorem~\ref{Thilbert}: the metric in ${\bf L}_\mathcal{H} \cong
LPM_{\mathcal{H}}(\beps)$ does not extend to the larger abelian group
${\bf L}^\C \cong LPM^\C(\beps)$. Nevertheless, as we presently explain,
not only can $LPM^\C(\beps)$ (and hence each subgroup $LPM^\F(\beps)$) be
given a translation-invariant metric, it in fact \textit{embeds into a
Banach space!} See Theorem~\ref{Trealclosed}, over more general fields.

\subsection{Real-closed fields}

In this part, we venture out beyond the complex plane. One can carry out
the Cholesky decomposition over certain fields outside of $\C$: namely,
the \textit{real-closed fields} $\mathbb{E}$~\cite{BCR} (which include
$\R$). In this setting, akin to above, one can work over the algebraic
closure $\F = \overline{\mathbb{E}} = \mathbb{E}[\sqrt{-1}]$, or more
generally, $\F$ the algebraic closure of an arbitrary subfield
$\mathbb{E}' \subseteq \mathbb{E}$ -- or even the iterated closure under
positive square roots of the positive elements. Over such fields $\F$:
\begin{enumerate}
\item One can define the cones $LPM_n^\F(\epsilon)$ and
$TPM_n^\F(\epsilon)$.
\item The Cholesky decomposition algorithm in Theorem~\ref{Tlpm} goes
through, and we have a family of bijections $\Phi_{B_\epsilon} : {\bf
L}_n^\F \to LPM^\F_n(\epsilon)$, one for each $B_\epsilon \in
LPM^\F_n(\epsilon)$.
One also has the linear and nonlinear maps~\eqref{Ediffeos} relating
$LPM^\F_n(\epsilon)$ to $TPM^\F_n(\epsilon)$ and $TPM^\F_n(\epsrev)$,
respectively.
\item The map $L \circledcirc K := \floo{L} + \floo{K} + \D(L)\D(K)$
again makes ${\bf L}^\F_n$ into an abelian group.
\end{enumerate}

However, (to our knowledge) there is no metric naturally associated with
a real-closed field! So the Riemannian manifold and abelian Lie group
structures do not go through. Nor is it clear how to take ``real scalar
multiples'', because e.g.\ $\D(L)^\pi$ is not defined for general $L$.

Despite these absences, one can still show as above that Cholesky space
${\bf L}^\F_n$ sits inside a Banach space. Hence so does the isomorphic
group $LPM^\F_n(\epsilon)$ (via $\Phi_{\D_\epsilon}$),  and hence the
larger union-groups:

\begin{theorem}\label{Trealclosed}
As above, let $\star \in \{ n, 00 \}$, and fix a sign sequence $\beps \in
\{ \pm 1 \}^{\star'}$, with $\star' = n$ if $\star = n$ and $\star' =
\infty$ otherwise.
Also let $\mathbb{E}$ be a real-closed field (e.g.\ $\R$) and let $\F$ be
any subfield of $\overline{\mathbb{E}} = \mathbb{E}[\sqrt{-1}]$ such that
$\F \cap (0,\infty)$ is closed under taking positive square roots.
\begin{enumerate}[$(1)$]
\item Then one has a tower of groups under $(\circledcirc, \Id_{\star'},
L_\circledcirc^{-1})$, with each pair ${\bf L}^\F_\star,
LPM^\F_\star(\beps)$ indeed in bijection via the Cholesky factorization
map $\Phi_{\D_{\beps}}^{\pm 1}$:
\begin{equation}\label{Erealclosed}
[ \, {\bf L}^\F_1 \cong LPM^\F_1((\epsilon_1)) \, ] \subset 
[ \, {\bf L}^\F_2 \cong LPM^\F_2((\epsilon_1,\epsilon_2)) \, ] \subset
\cdots \subset
[ \, {\bf L}^\F_{00} \cong LPM^\F_{00}(\beps) \, ] \subset
[ \, {\bf L}^\F \cong LPM^\F(\beps) \, ].
\end{equation}

\item The group ${\bf L}^\F \cong LPM^\F(\beps)$ additively embeds inside
a real Banach space $\mathbb{B}$, hence inherits its norm. One can
therefore apply Banach space probability~\cite{LT} to each of its
subgroups $G$, e.g.\ $G$ in~\eqref{Erealclosed}. Thus the expectation of
a $G$-valued random variable makes sense (and it lives in $\mathbb{B}$).
\end{enumerate}
\end{theorem}

\begin{proof}
The first part is similar to the special case of $\mathbb{E} = \R$ and
$\overline{\mathbb{E}} = \C$. For the second, we invoke a result from a
recent Polymath project~\cite{Polymath}:
\textit{a group is abelian and torsionfree if and only if it embeds
additively inside a real Banach space $\mathbb{B}$.}
(Note that the separability of $\mathbb{B}$ is not assured.)
Thus it suffices to check that ${\bf L}^\F \cong LPM^\F(\beps)$ is
torsionfree. But if $L \in {\bf L}^\F$ and $0 < k \in \mathbb{Z}$, then
\[
k \cdot L = k \floo{L} + \D(L)^k = \Id_{\star'}.
\]
Since ${\rm char}(\F) = 0$, $\floo{L} = 0$; and since $\F$ is totally
ordered, elements in $(0,1)$ and $(1,\infty)$ cannot have $k$th power
$1$. So $\D(L) = \Id_{\star'}$, as desired.
\end{proof}

\subsection{Probability inequalities for Cholesky or LPM-valued random
matrices}\label{Sprob}

Earlier in this section and the previous one, we saw three kinds of
abelian metric groups:
\begin{itemize}
\item Subgroups of $(LPM_{\mathcal{H}}^\F(\beps), \circledast) \cong
({\bf L}^\F_{\mathcal{H}}, \circledcirc)$, with $\F \subseteq \C$ and $\F
\cap (0,\infty)$ closed under positive square roots. By
Theorem~\ref{Ttower}, these metric groups embed additively and
isometrically inside the Hilbert space $LPM_{\mathcal{H}}^\C(\beps)$. As
a result, standard Euclidean and Hilbert (hence Banach) space probability
results apply to these groups.

\item Subgroups of $(LPM^\F(\beps), \circledast) \cong ({\bf L}^\F,
\circledcirc)$, with $\mathbb{E}$ a real-closed field and $\F \subseteq
\mathbb{E}[\sqrt{-1}]$ having the same $\sqrt{\cdot}$-closure property.
These groups additively embed inside Banach spaces, so a comprehensive
probability theory again applies to them~\cite{LT} (including
expectations).

\item Then there are the abelian metric subgroups of
$(LPM^\F_{\mathcal{H}}, \boxdot)$ studied in Theorem~\ref{TbiggroupsF}
(with the parallel group structure $\boxdot$, and with $\F \subseteq
\C$).
\end{itemize}

The third class of groups have nontrivial torsion elements, hence do not
embed in any Banach space. As we now explain, a mass of probability
inequalities nevertheless applies to these groups (and to the other two
group-classes too). The following result expands on Theorem~\ref{Tprob}.

\begin{theorem}\label{Tprob2}
Fix $\beps \in \{ \pm 1 \}^\infty$, and let $G$ be any of the following
abelian groups:
\begin{itemize}
\item A subgroup of $(LPM^\F_{\mathcal{H}}(\beps), \circledast)$, with
$\F \subseteq \C$ and $\F \cap (0,\infty)$ closed under positive square
roots.
\item A separable subgroup of $(LPM^\F(\beps), \circledast)$ equipped
with a norm, where $\F$ is as in Theorem~\ref{Trealclosed}.
\item A separable subgroup of $(LPM^\F_{\mathcal{H}}, \boxdot)$, with the
metric $d_p$ for some $p \in [1,\infty]$. Here, $\F \subseteq \C$ and $\F
\cap (0,\infty)$ is closed under positive square roots.
\end{itemize}

Now fix $z_1 \in G$ and independent random variables $X_1, \dots,
X_n \in L^0(\Omega,G)$, with $\mu$ a probability measure on $\Omega$. Set
$S_k := X_1 + \cdots + X_k$, with `$+$' denoting either $\circledast$ or
$\boxdot$. Also define
\begin{equation}
U_n := \max_{1 \leq k \leq n} d_G(z_1, S_k), \qquad
M_n := \max_{1 \leq k \leq n} d_G(\Id_G, X_k),
\end{equation}
where $d_G$ is the bi-invariant metric on $G$.
Then the following stochastic inequalities hold in $G$:
\begin{enumerate}[$(1)$]
\item \underline{\em Mogul'skii inequalities.}
Fix $a,b,c \in [0,\infty)$. If $1 \leq m \leq n$, then:
\begin{align}
& \bp{\min_{m \leq k \leq n} d_G(z_1, S_k) \leq a} \cdot \min_{m
\leq k \leq n} \bp{ d_G(S_k, S_n) \leq b} \leq \bp{ d_G(z_1, S_n)
\leq a + b},\\
& \bp{\max_{m \leq k \leq n} d_G(z_1, S_k) \geq a} \cdot \min_{m
\leq k \leq n} \bp{ d_G(S_k, S_n) \leq b} \leq \bp{ d_G(z_1, S_n)
\geq a - b}.
\end{align}

\item \underline{\em Ottaviani--Skorohod inequality.}
Fix $\alpha, \beta \in (0,\infty)$. Then:
\begin{equation}
\bp{\max_{1 \leq k \leq n} d_G(z_1, S_k) \geq \alpha  +\beta}
\cdot \min_{1 \leq k \leq n} \bp{ d_G(S_k, S_n) \leq \beta} \leq
\bp{ d_G(z_1, S_n) \geq \alpha}.
\end{equation}

\item \underline{\em L\'evy--Ottaviani inequality.} 
For $a \geq 0$, define
$p_a := \max_{1 \leq k \leq n} \bp{d_G(z_1, S_k) > a}$.
Then
\begin{equation}
\bp{U_n > a_1 + \dots + a_l} \leq \sum_{i=2}^l p_{a_i} + p'_l, \qquad
\forall l \geq 2, \ a_1, \dots, a_l \geq 0,
\end{equation}
where $p'_l := p_{a_1}$ if $l$ is odd, and $p'_l := \max_{1 \leq k \leq
n} \bp{d_G(S_k, S_n) > a_1}$ if $l$ is even.

\item \underline{\em Hoffmann-J{\o}rgensen inequality.} 
Fix integers $0 < k, n_1, \dots, n_k \in \mathbb{Z}$ and nonnegative
scalars $t_1, \dots, t_k, s \in [0,\infty)$, and define
$I_0 := \{ 1 \leqslant i \leqslant k : \bp{U_n \leqslant t_i}^{n_i -
\delta_{i1}} \leqslant \frac{1}{n_i!} \}$,
where $\delta_{i1}$ denotes the Kronecker delta.
Now if $\sum_{i=1}^k n_i \leqslant n+1$, then
\begin{align}\label{Ehj}
&\ \bp{U_n > (2 n_1 - 1) t_1 + 2 \sum_{i=2}^k n_i t_i + \left(
\sum_{i=1}^k n_i - 1 \right) s}\\
\leqslant &\ \bp{M_n > s} +
\bp{U_n \leqslant t_1}^{{\bf 1}_{1 \notin I_0}} \prod_{i \in
I_0} \bp{U_n > t_i}^{n_i} \prod_{i \notin I_0} \frac{1}{n_i!} \left(
\frac{\bp{U_n > t_i}}{\bp{U_n \leqslant t_i}} \right)^{n_i}.\notag
\end{align}
For a stronger version, in terms of the order statistics of the $Y_j$,
see~\cite[Theorem~A]{KR1}.
\end{enumerate}
\end{theorem}

As explained in~\cite{Khare,KR1}, Theorem~\ref{Tprob2}(4) firstly
goes beyond the previously known versions, even in the most special case
of $\R$;
it moreover strengthens and unifies previous results in the Banach space
literature by Kahane, Hoffmann-J{\o}rgensen, Johnson--Schechtman,
Klass--Nowicki, and Hitczenko--Montgomery-Smith;
and it also extends these results from Banach spaces to arbitrary metric
semigroups with a bi-invariant metric -- in particular, $G$.
The other inequalities also generalize and strengthen previously known
variants in the literature, beyond Banach spaces.

\begin{proof}
These results are, in a sense, black boxes. Namely, they were shown
in~\cite{KR1,KR2} in the more general setting of (a)~separable metric
semigroups with a bi-invariant metric; and in~\cite{Khare} for the
strictly more general setting of (b)~separable metric monoids with a
left-invariant metric. As every subgroup $G$ in Theorem~\ref{Tprob2} fits
into~(a) and hence~(b), the results apply here. Moreover, since $d_G$ is
translation-invariant, inside the expression $d_G(z_1, z_0 W)$ in all
results above we can -- and do -- replace $z_1$ by $z_1-z_0 := z_1
\circledast (z_0)_\circledast^{-1}$ and $z_0$ by the identity $\Id_G$, in
the abelian group $(G, \circledast)$ in the cited versions. Similarly by
$z_1 \boxdot (z_0)^{-1}_\boxdot$ in $(G,\boxdot)$.
\end{proof}

The proof of the next result uses the aforementioned Ottaviani--Skorohod
inequality.

\begin{theorem}[L\'evy's Equivalence]\hfill
\begin{enumerate}[$(1)$]
\item Fix $n \geq 1$ and $\epsilon \in \{ \pm 1 \}^n$, and let $\F = \R$
or $\C$. Let $G = LPM_n^\F(\epsilon)$ or $TPM_n^\F(\epsilon)$, and let
$X_k$ be independent $G$-valued random variables. Then the partial sums
$S_m = X_1 \circledast \cdots \circledast X_m$ converge almost surely to
a $G$-valued random variable $X$, if and only if they converge to $X$ in
probability.

\item The same holds for variables $X_k$ taking values in $(LPM^\F_n,
\boxdot)$ or $(TPM^\F_n, \boxdot)$.

\item If either sequence $S_n$ does not converge in the above manner,
then it diverges almost surely.
\end{enumerate}
\end{theorem}

This result was also shown in~\cite{Khare,KR2} in the general
settings~(a) and~(b) (as mentioned in the proof of Theorem~\ref{Tprob2})
-- for $G$ a complete separable metric semigroup with a bi-invariant
metric. But in the present setting, $G$ is a complete separable metric
group, in which case the result was already proved earlier, by
Tortrat~\cite{Tortrat}. It too extends previous results in the Banach
space probability literature, by It\^o--Nisio and by
Hoffmann-J{\o}rgensen--Pisier.

\section{(Inverse) Wishart and Cholesky-normal densities on LPM and TPM
cones}\label{Sstat}

Throughout this final section, we will work over $\F = \R$; thus, we
write $LPM_n(\epsilon) = LPM_n^\R(\epsilon)$ etc. Now recall
Definition~\ref{Dwishart}, which explained how to transfer a probability
density from the cone $PD_n$ to any $LPM_n(\epsilon)$ or
$TPM_n(\epsilon)$. We use this to take a closer look at the (inverse)
Wishart densities, and in particular, prove Theorem~\ref{Twishart}.

We begin with the classical Wishart distribution~\cite{Wishart} (see also
Fisher~\cite{Fisher} when $n=2$). Given integers $N \geq n \geq 1$ and a
fixed matrix $\Sigma \in PD_n$, one says a random positive definite
matrix ${\bf M_1} \sim W_n(\Sigma, N)$ if it has density
\begin{equation}\label{Ewishart}
f_{\Sigma,N}({\bf M_1}) := \frac{1}{2^{nN/2} \Gamma_n(N/2)
\det(\Sigma)^{N/2}} \det({\bf M_1})^{(N-n-1)/2} \exp \left( - {\rm
tr}(\Sigma^{-1} {\bf M_1})/2 \right)
\end{equation}
for ${\bf M_1} \in PD_n$, and $0$ otherwise, with $\Gamma_n$ the
multivariate gamma function. Now Definition~\ref{Dwishart} transfers this
to multivariate cones -- which we explicitly write out to set notation:

\begin{defn}
A random matrix ${\bf M} \sim W_n^{LPM}(\epsilon,\Sigma,N)$ if it has
density supported on $LPM_n^\R(\epsilon)$ and given there by
\begin{equation}\label{ElpmWishart}
f^{LPM}_{\epsilon,\Sigma,N}(L \D_\epsilon L^T) :=
f_{\Sigma,N}(L L^T),
\qquad \text{i.e.,} \qquad f^{LPM}_{\epsilon,\Sigma,N}({\bf M}) :=
f_{\Sigma,N}( \Phi_{\Id_n} \circ \Phi_{\D_\epsilon}^{-1}({\bf M})).
\end{equation}

Similarly, the Wishart density $W_n^{TPM}(\epsilon, \Sigma, N)$ is given
by its density function:
\begin{equation}\label{ETPMdensity}
f_{\epsilon,\Sigma,N}^{TPM}(K^T \cdot P_n \D_\epsilon P_n \cdot K) :=
f_{\Sigma,N}(K^T K) = f_{\Sigma,N}(\Phi^{\Id_n} \circ (\Phi^{P_n
\D_\epsilon P_n})^{-1}(B)),
\end{equation}
where $P_n$ is as in Theorem~\ref{Ttpm}, $B := K^T \cdot P_n \D_\epsilon
P_n \cdot K$, and the map $\Phi^{C_\epsilon}$ is as in
Theorem~\ref{Ttpm}.
\end{defn}

With these definitions at hand, we proceed towards
Theorem~\ref{Twishart}. In it, since $\Phi_{\D_\epsilon}$ sends a matrix
$L$ (with $\binom{n+1}{2}$ degrees of freedom) to a symmetric matrix $A
:= L \D_\epsilon L^T$ (with the same degrees of freedom), we order the
input and output coordinates in lexicographic order:
\begin{align}\label{Elexico}
\begin{aligned}
(l_{11}; \ l_{21}, l_{22}; &\ \ \dots; \ l_{n1}, \dots, l_{nn}),\\
(a_{11}; \ a_{21}, a_{22}; &\ \ \dots; \ a_{n1}, \dots, a_{nn}).
\end{aligned}
\end{align}
We also use its reverse -- the ``reverse-lexicographic order''. Now we
perform Jacobian calculations.

\begin{prop}\label{Pjacobian}
Set $\epsilon_0 := 1$. The Jacobian of $\Phi_{\D_\epsilon}$ at any $L \in
{\bf L}_n^\R$ is lower triangular in the lexicographic
order~\eqref{Elexico}, and its determinant equals
\begin{equation}\label{Ejacobian}
\det \left( \frac{\partial \Phi_{\D_\epsilon}(L)_{ij}}{\partial l_{i'j'}}
\right)_{\substack{1 \leq j \leq i \leq n\\1 \leq j' \leq i' \leq n}} = \
2^n \prod_{j=1}^n (l_{jj} \epsilon_{j-1} \epsilon_j)^{n+1-j}.
\end{equation}

Similarly, given $L' = (l'_{ij}) \in {\bf L}_n^\R$, the Jacobian of
$\Phi^{\D_{\epsrev}}$ at $\rev{L'} \in {\bf L}_n^\R$ is upper triangular,
with
\begin{equation}
\det \left( \frac{\partial \Phi^{\D_{\epsrev}}(\rev{L'})_{i'j'}}{\partial
\rev{L}} \right)_{\substack{1 \leq j' \leq i' \leq n\\1 \leq j \leq i
\leq n}} = \ 2^n \prod_{j=1}^n (l'_{n+1-j,n+1-j} \epsilon_{j-1}
\epsilon_j)^{n+1-j},
\end{equation}
where $\partial \rev{L}$ means the reverse-lexicographic order.
Moreover, the two Jacobians are related via:
\begin{equation}\label{Ejacobians}
\left. \frac{\partial \Phi^{\D_{\epsrev}} }{\partial \rev{L}}
\right|_{\rev{L'}} \ = \ P_{\binom{n+1}{2}} \cdot \left.
\frac{\partial \Phi_{\D_\epsilon}}{\partial L} \right|_{L'} \cdot
P_{\binom{n+1}{2}}.
\end{equation}
\end{prop}

\begin{proof}
We begin with~\eqref{Ejacobian}. For this, we claim that for any $n \geq
i \geq j \geq 1$, the expression $a_{ij} = (L \D_\epsilon L^T)_{ij}$ does
not depend on $l_{ik}, k>j$ or on $l_{hj}, h>i$. This is clear by direct
computation:
\begin{equation}\label{Ecompute}
a_{ij} = \sum_{k=1}^j l_{ik} \epsilon_{k-1} \epsilon_k l_{jk}.
\end{equation}

Thus the Jacobian on the left side of~\eqref{Ejacobian} is lower
triangular under the lexicographic order~\eqref{Elexico}. Moreover,
$\partial a_{ij} / \partial l_{ij}$ equals $\epsilon_{j-1} \epsilon_j
l_{jj}$ if $i>j$, and $2 \epsilon_{j-1} \epsilon_j l_{jj}$ if $i=j$.
Multiplying these yields~\eqref{Ejacobian}.

The next assertion is shown similarly. Finally,~\eqref{Ejacobians}
involves straightforward, albeit somewhat tedious, bookkeeping. We omit
the details.
\end{proof}

With Proposition~\ref{Pjacobian} at hand, we explain our final main
result.

\begin{proof}[Proof of Theorem~\ref{Twishart}]\hfill
\begin{enumerate}[(1)]
\item We begin by showing~\eqref{Echangeofvar}.
(We do not prove the TPM counterpart of~\eqref{Echangeofvar} as it is
similar.)
First, for a random variable ${\bf M} = L \D_\epsilon L^T$, define
\[
{\bf M_1} :=  L L^T = \Phi_{\Id_n} \circ \Phi_{\D_\epsilon}^{-1}({\bf M})
\in PD_n.
\]
Then the right-hand side of~\eqref{Echangeofvar} equals
$\displaystyle \int_{{\bf M_1} \in \Phi_{\Id_n} \circ
\Phi_{\D_\epsilon}^{-1}(\mathscr{A})} f_Q({\bf M_1})\, d {\bf M_1}$,
where $d {\bf M_1}$ is the Lebesgue measure on the sub-cone $PD_n$ of $n
\times n$ real symmetric matrices.

Now change variables from ${\bf M_1}$ to ${\bf M}$. This changes the
domain of integration to $\mathscr{A}$, and the integrand can be
rewritten (by definition) as $f^{LPM}_{\epsilon,Q}({\bf M})$. As the
absolute value of the Jacobian of $\Phi_{\Id_n} \circ
\Phi_{\D_\epsilon}^{-1}$ (and of its inverse) is $1$ by
Proposition~\ref{Pjacobian} (and the Chain Rule), we
obtain~\eqref{Echangeofvar}.

\item This is essentially the preceding part, specialized to $Q =
W_n(\Sigma,N)$.

\item We omit these details.

\item We next turn to~\eqref{Ewishartrev}, where we only prove the
forward implication, as the proof of its converse is similar. Note that
the reversal map is Lebesgue measurable; thus by~\eqref{Echangeofvar} for
$Q = W_n(\Sigma,N)$, and for measurable $\mathscr{B} \subseteq
TPM_n(\epsilon)$,
\begin{align}\label{ETPMwishart}
&\ \bbp \left( \rev{\bf M} = \rev{L}^T \D_{\epsrev} \rev{L} \in
\mathscr{B} \ \middle| \ {\bf M} \sim W_n^{LPM}(\epsilon, \Sigma, N)
\right)
= \bbp( {\bf M} = L \D_\epsilon L^T \in P_n \mathscr{B} P_n)\notag\\
= &\ \int_{{\bf M_1} = L L^T \in \Phi_{\Id_n} \circ
\Phi_{\D_\epsilon}^{-1}( P_n \mathscr{B} P_n)} f_{\Sigma,N}({\bf
M_1}) \, d {\bf M_1}\notag\\
= &\ \frac{2^{-nN/2}}{\Gamma_n(N/2)} \int_{{\bf M_1} = L L^T \in
\Phi_{\Id_n} \circ \Phi_{\D_\epsilon}^{-1}( P_n \mathscr{B} P_n)}
\frac{\det({\bf M_1})^{(N-n-1)/2}}{\det(\Sigma)^{N/2}} \exp \left( {\rm
tr} (- \Sigma^{-1} {\bf M_1})/2 \right) \, d {\bf M_1}.
\end{align}

Change variables from ${\bf M_1} = L L^T$ to
$\overset{\longleftarrow}{\bf M_1} = P_n L L^T P_n = \rev{L^T} \rev{L}$.
Adjoin to the left of the commuting square~\eqref{Erevsq} the analogous
square (reversed) for $\epsilon = {\bf 1}_n$. As all maps in these
squares are bijections by Theorems~\ref{Tlpm} and~\ref{Ttpm}, we can
change variables in the domain in the above integral:
\[
{\bf M_1} = \Phi_{\Id_n}(L) \in \Phi_{\Id_n} \circ
\Phi_{\D_\epsilon}^{-1}( P_n \mathscr{B} P_n) \quad \Longleftrightarrow
\quad \Phi_{\D_\epsilon}(L) \in P_n \mathscr{B} P_n.
\]
Applying the reversal map and using~\eqref{Erevsq}, this is if and only
if $\Phi^{\D_{\epsrev}}(\rev{L}) = \rev{\bf M_1} \in \mathscr{B}$.

This changes the domain. The Jacobian of the transformation ${\bf M_1}
\mapsto \overset{\longleftarrow}{\bf M_1}$ is a permutation matrix, hence
unimodular. Finally, we change the density itself: $\det {\bf M_1} = \det
\overset{\longleftarrow}{\bf M_1}$, and
\[
{\rm tr}(-\Sigma^{-1} {\bf M_1}/2) = 
{\rm tr}(-P_n \Sigma^{-1} {\bf M_1} P_n/2) = 
{\rm tr}(- \rev{\Sigma}^{-1} \rev{\bf M_1} /2).
\] This suggests the Wishart density with parameters $N,\rev{\Sigma}$.
And indeed, the remaining factor is $\det(\Sigma)^{-N/2} =
\det(\rev{\Sigma})^{-N/2}$.

Thus, we can now continue the calculation in~\eqref{ETPMwishart}, via
changing variables $: {\bf M_1} \to \rev{\bf M_1}$:
\begin{align*}
&\ \bbp_{{\bf M} \sim W_n^{LPM}} \left( \rev{\bf M} = \rev{L}^T
\D_{\epsrev} \rev{L} \in \mathscr{B} \right)\notag\\
= \cdots = &\ \frac{2^{-nN/2}}{\Gamma_n(N/2)} \int_{\rev{\bf M_1} =
\rev{L^T} \rev{L} \in \Phi^{\Id_n} \circ (\Phi^{\D_{\epsrev}})^{-1}( 
\mathscr{B} )} \frac{\det(\rev{\bf
M_1})^{(N-n-1)/2}}{\det(\rev{\Sigma})^{N/2}} \exp \left( {\rm tr} (-
\rev{\Sigma}^{-1} \rev{\bf M_1}) /2 \right) \, d
\overset{\longleftarrow}{\bf M_1}.
\end{align*}
As the right-hand side is the push-forward to $TPM_n^\R(\epsilon)$ of the
Wishart density with parameters $\rev{\Sigma},N$, and the equality holds
for all events $\mathscr{B} \subseteq TPM^\R_n(\epsilon)$, the proof is
complete.

\item The key observation is that if $g$ is any function, and ${\bf M} =
U D U^T$ where $U$ is orthogonal and $D$ is real diagonal, then $\rev{\bf
M} = P_n M P_n = (P_nU) D (P_nU)^T$. Now we compute, via using e.g.\
\cite[Section~1.2]{Higham-book} and since all matrices below in the
arguments of $g$ are real symmetric:
\[
g(\rev{\bf M}) = (P_nU) g(D) (P_nU)^T = P_n (U g(D) U^T) P_n = P_n g({\bf
M}) P_n = \overset{\leftarrow}{g}({\bf M}).
\]

If $\mathbb{E}[g({\bf M})]$ exists for ${\bf M} \sim
W_{\epsilon,n}^{LPM}(\Sigma,N)$, then compute via first
using~\eqref{Ewishartrev} and then changing variables $\rev{\bf M_1}
\leadsto {\bf M_1}$ opposite to the previous part (but still with
unimodular Jacobian):
\begin{align*}
\mathbb{E}[g(\rev{\bf M})] = &\ \int_{\rev{\bf M_1} \in PD_n} g(\rev{\bf
M}) f_{\rev{\Sigma},N}(\rev{\bf M_1}) \, d \rev{\bf M_1}
= \int_{{\bf M_1} \in PD_n} (P_n g({\bf M}) P_n) f_{\Sigma,N}({\bf M_1})
\, d {\bf M_1}\\
= &\ P_n \cdot \int_{{\bf M_1} \in PD_n} g({\bf M}) f_{\Sigma,N}({\bf
M_1}) \, d {\bf M_1} \cdot P_n = P_n \mathbb{E}[ g({\bf M}) ] P_n.
\end{align*}
This shows the integrability of $\mathbb{E}[ g(\rev{\bf M}) ]$ and
completes the proof.\qedhere
\end{enumerate}
\end{proof}

Given Theorem~\ref{Twishart}, one can ``glue together'' probability
distributions on cones for different $\epsilon$, via use of a weight-set
$(w_\epsilon)_{\epsilon \in \{ \pm 1 \}^n}$ to generate probability
distributions supported on $LPM_n^\R$ or on $TPM_n^\R$ -- which, we
remind the reader, are open and dense in all real symmetric matrices. We
will revisit this in Proposition~\ref{Pcloning} when talking about the
inertia of matrices in LPM/TPM cones.

\begin{remark}\label{Rnatural}
A natural question, given Theorem~\ref{Twishart}, again involves changing
variables between the ``usual'' Wishart density and one on a LPM/TPM cone
-- or more generally, between two probability densities. One not only
changes the differential via the Jacobian (which we did in
Proposition~\ref{Pjacobian}), but also the domain (via the map
$\Phi_{\Id_n} \circ \Phi_{\D_\epsilon}^{-1}$), and moreover, the density
itself. For this last task, one seeks to write $L L^T$ in terms of $A :=
L \D_\epsilon L^T$ -- or more generally:
\[
\widetilde{A} = L \D_\delta L^* \in LPM_n^\C(\delta)
\quad \text{in terms of} \quad A = L \D_\epsilon L^*, \qquad
\text{for } \delta, \epsilon \in \{ \pm 1 \}^n.
\]
This is given by first applying Algorithm~\ref{AlgCholesky} to $A$ and
obtaining $L$; followed by $L \mapsto \widetilde{A} = L \D_\delta L^T$.
Explicitly, we define sequentially for $k=1,\dots,n$:
\begin{equation}
\alpha_j^{(k)} := a_{jk} - \sum_{i=1}^{k-1} \frac{\alpha_j^{(i)}
\overline{\alpha_k^{(i)}} }{l_{ii}^2} \epsilon_{i-1} \epsilon_i, \qquad
\forall k \leq j \leq n,
\end{equation}
where $\epsilon_0 = 1$, and $l_{ii}^2 = \epsilon_{i-1} \epsilon_i \det
A_{[i][i]} / \det A_{[i-1][i-1]}$ is as in~\eqref{Eqsquare} or
Algorithm~\ref{AlgCholesky} with $B_\epsilon = \D_\epsilon$. Then
$\widetilde{A}$ is given by:
\begin{equation}\label{Efinal}
\widetilde{a}_{jk} = \sum_{i=1}^k \frac{\alpha_j^{(i)}
\overline{\alpha_k^{(i)}} }{l_{ii}^2} \delta_{i-1} \delta_i, \qquad
\forall 1 \leq k \leq j \leq n.
\end{equation}
Both $\alpha_j^{(k)}$ and $\widetilde{a}_{jk}$ are computed by induction
on $k$ via~\eqref{Ecompute}, or rather, its complex
generalization:
\begin{equation}
a_{jk} = \sum_{i=1}^k l_{ji} \epsilon_{i-1} \epsilon_i \overline{l_{ki}},
\qquad \forall 1 \leq k \leq j \leq n.
\end{equation}
\end{remark}

\subsection{Additional distributions}

We briefly discuss other probability distributions supported on $PD_n$
and their adaptations to LPM/TPM cones. The first was the Wishart family,
above.

\subsubsection{}
As a second example, one can define the \textit{complex Wishart
density}~\cite{Goodman1,Goodman2} on the cones $LPM_n^\C(\epsilon),
TPM_n^\C(\epsilon)$. We omit the definitions.

\subsubsection{}
The third distribution we mention is the inverse Wishart distribution. If
${\bf M_1} \sim W_n(\Sigma,N)$ then one can apply the change of variables
${\bf M_1} \mapsto {\bf X_1} := {\bf M_1}^{-1}$, and the resulting
density (including the Jacobian determinant) is written in terms of
$\Omega := \Sigma^{-1}$ as:
\begin{equation}
g_{\Omega,N}({\bf X_1}) := \frac{\det(\Omega)^{N/2}}{2^{nN/2}
\Gamma_n(N/2)} \det({\bf X_1})^{-(N+n+1)/2} \exp(-{\rm tr}(\Omega {\bf
X_1}^{-1})/2), \qquad {\bf X_1} \in PD_n.
\end{equation}

Recall by Proposition~\ref{Ptpm} that if ${\bf M}$ is supported on
$LPM_n^\R(\epsilon)$ then ${\bf X} = {\bf M}^{-1}$ is supported on
$TPM_n^\R(\epsrev)$. Thus, Definition~\ref{Dwishart},
Theorem~\ref{Twishart}, Proposition~\ref{Ptpm}, and~\eqref{Esanity}
combine to yield:

\begin{prop}
Fix integers $N \geq n \geq 1$ and a sign pattern $\epsilon \in \{ \pm 1
\}^n$.
\begin{enumerate}[$(1)$]
\item If ${\bf M} \sim W_n(\epsilon,\Sigma,N)$ on $LPM_n^\R(\epsilon)$,
then ${\bf M}^{-1} \sim W_n^{-1}(\epsrev,\Sigma^{-1},N)$ on
$TPM_n^\R(\epsrev)$.
\item The TPM analogue also holds.
\end{enumerate}
As a consequence,
$P_n {\bf M}^{-1} P_n = (\rev{\bf M})^{-1} \sim
W_n^{-1}(\epsrev,(\rev{\Sigma})^{-1},N)$ on $LPM_n^\R(\epsrev)$ in
part~(1).
\end{prop}

\subsubsection{}
The fourth family is that of PD-matrix-variate lognormal distributions of
type I~\cite[Section~2.4]{Armin}. This uses the result
in~\cite{VectorSpace} that the addition and scalar multiplications
\[
A' \odot B' := \exp( \log(A') + \log(B') ), \qquad \alpha .' A' :=
\exp(\alpha \log(A'))
\]
make $PD_n$ isomorphic to an $\R$-vector space (see
Remark~\ref{RVectorSpace}). Given this, one can transfer using
Definition~\ref{Dwishart} and Theorem~\ref{Twishart} this distribution to
the cones $LPM_n(\epsilon), TPM_n(\epsilon)$. One can then also create
densities via ${\bf M} \leadsto {\bf M}^{-1}$, $\rev{\bf M}$, or
$(\rev{\bf M})^{-1}$.

\subsubsection{Cholesky-normal densities}

Our fifth candidate is a novel family of densities defined on the PD cone
itself. This is a parallel construction to the preceding type-I
lognormal densities; for its definition we need the Euclidean space
isomorphism map $\eta$ defined in Theorem~\ref{Tinnerproduct}.

Fix $n \geq 1$ and suppose $A_1, \dots, A_m$ are i.i.d.\ $PD_n$-valued
random matrices, such that
\[
\mathbb{E} \left[ \eta \circ \Phi_{\Id_n}^{-1}( A_j ) \right]
= \eta \circ \Phi_{\Id_n}^{-1}(M_\circ),
\qquad {\rm Cov} \left( \eta \circ \Phi_{\Id_n}^{-1}( A_j ) \right) =
\widetilde{\Sigma} \in PD_{\binom{n+1}{2}}
\]
for some $M_\circ \in PD_n$.
Then by SLLN and CLT, their log-Cholesky average/barycentre satisfies:
\begin{align}
\begin{aligned}
\widehat{A}_m := \Phi_{\Id_n} \circ \eta^{-1} \left( \frac{1}{m}
\sum_{j=1}^m \eta \circ \Phi_{\Id_n}^{-1}(A_j) \right) &\
\overset{a.s.}{\longrightarrow} M_\circ, \\
\sqrt{m} \left( \eta \circ \Phi_{\Id_n}^{-1}(\widehat{A}_m) - \eta \circ
\Phi_{\Id_n}^{-1}(M_\circ) \right) &\ \Longrightarrow N({\bf 0},
\widetilde{\Sigma}).
\end{aligned}
\end{align}

Thus, we introduce:

\begin{defn}\label{Dcholeskynormal}
A $PD_n$-valued random variable ${\bf M_1}$ is said to have a
\textit{Cholesky-normal distribution} with parameters $M_\circ \in PD_n$
and $\widetilde{\Sigma} \in PD_{\binom{n+1}{2}}$, if $\eta \circ
\Phi_{\Id_n}^{-1}({\bf M_1}) \sim N( \eta \circ
\Phi_{\Id_n}^{-1}(M_\circ), \widetilde{\Sigma})$.
\end{defn}

\begin{remark}
Using Definition~\ref{Dwishart}, we transfer this density to every cone
$LPM_n(\epsilon), TPM_n(\epsilon)$.
\end{remark}

\subsubsection{}

The notion of the canonical geometric mean leads Schwartzman
in~\cite[Section~3.3]{Armin} to propose PD-matrix-variate lognormal
distributions of type II -- which one can again transfer to all cones
$LPM_n(\epsilon), TPM_n(\epsilon)$.

\medskip
Given the extensive work into studying the Wishart family on the cone
$PD_n$, it will be interesting to analyze in depth the above densities
over parallel and larger cones, with multivariate analysis, random matrix
theory, and statistical applications (e.g.\ \cite{Olkin,Armin}) in mind.

\subsection{Inertia and LPM cones}\label{Sinertia}

We conclude with some inertia results that are motivated by modern
considerations. An area of much recent interest and activity involves
embedding data and graphs/discrete structures in hyperbolic manifolds,
and it features in multiple applied fields (see the ``modern'' part of
Section~\ref{Smotivations}). This follows classical explorations of
embedding metric spaces into hyperbolic manifolds, by Krein with
Iokhvidov~\cite{IK} following Krein's announcement~\cite{Krein}.

Thus, suppose we sample points $x(1), \dots, x(n)$ -- say under some
distribution -- from hyperbolic space $\ell^2_\R$ (with a different,
indefinite inner product $[\cdot,\cdot]$), or from the ``unit sphere'' in
it, termed \textit{Lobachevsky space}:
\[
\mathcal{L} := \{ x = (x_0, x_1, \dots) \in \ell^2_\R : x_0 > 0 \text{
and } [x,x] = 1 \}, \quad \text{where} \quad [x,y] := x_0 y_0 -
\sum_{j=1}^\infty x_j y_j.
\]

In the case of (positively curved) Euclidean and Hilbert spheres $S^r
\subset S^\infty$, the metric $d(x,y)$ for $x,y \in S^\infty$ is
recovered from Euclidean Gram matrices via: $\arccos \langle x, y
\rangle$. In a precise parallel, the hyperbolic metric equals ${\rm
arccosh} [x, y]$ on $\mathcal{L}$. Thus, modern hyperbolic data analysis
as well as classical hyperbolic metric embeddings both require working
with hyperbolic Gram matrices $([x(i),x(j)])_{i,j=1}^n$ -- which were
termed \textit{Lorentz--Gram matrices} by Loewner in~\cite{Loew65}. Now
classical results of Krein and others show the following
characterization: \textit{Lorentz--Gram matrices are precisely the real
symmetric matrices with diagonal entries $1$ and exactly one positive
eigenvalue.} If we further assume such $n \times n$ matrices are
nonsingular, then they have precisely $n-1$ negative eigenvalues.\medskip

This brings to our attention the invariant of \textit{inertia}: the
numbers of positive, negative, and zero eigenvalues. Restricting
ourselves to the open dense cones of Hermitian matrices with nonsingular
leading/trailing principal submatrices, our goal here is to partition
matrices with a fixed inertia into $LPM_n(\epsilon)$ or $TPM_n(\epsilon)$
cones.

We begin with notation. Given integers $0 \leq k \leq n$ and an arbitrary
subfield $\F \subseteq \C$, let $\In^\F_n(k)$ denote the matrices in
$LPM_n^\F$ with exactly $k$ negative eigenvalues (and all entries in
$\F$). For instance, $\In^\F_n(0) = PD^\F_n = LPM^\F_n({\bf 1}_n)$. Note
that since each such matrix is invertible and Hermitian, the remaining
$n-k$ values must be positive, and so the complete inertia is specified
by just $(n,k)$. We call $k$ the \textit{negative inertia} of the matrix.

Thus, the LPM-cone admits two partitions (though the second covers all of
$LPM_n^\F$ only when $\F \cap (0,\infty)$ is closed under positive square
roots):
\[
LPM_n^\F = \bigsqcup_{0 \leq k \leq n} \In_n^\F(k) = \bigsqcup_{\epsilon
\in \{ \pm 1 \}^n} LPM_n^\F(\epsilon).
\]
Are these two ``compatible'', i.e.\ is every factor in the first
partition equal to a union of ``entire pieces'' from the second? (Dually,
does every matrix in a fixed $LPM_n^\F(\epsilon)$ have the same inertia?)
The next result explores this property in detail.

\begin{theorem}\label{Tinertia}
Fix $n \geq 1$ and a subfield $\F \subseteq \C$.
\begin{enumerate}[$(1)$]
\item Given $\epsilon \in \{ \pm 1 \}^n$, every matrix in
$LPM_n^\F(\epsilon)$ has the same negative inertia, which equals
$\vartheta(1;\epsilon)$. Here, $\vartheta(1;\epsilon)$ denotes the number
of sign changes in the sequence $\epsilon_0 = 1, \epsilon_1, \epsilon_2,
\dots, \epsilon_n$.

\item The number of $LPM_n^\F(\epsilon)$ cones sitting inside
$\In_n^\F(k)$ is $\binom{n}{k}$ (so these add up to $2^n = \# \{ \epsilon
\}$). Thus if $\F \cap (0,\infty)$ is closed under positive square roots,
then $\In_n^\F(k) = \bigsqcup_{\vartheta(1;\epsilon) = k}
LPM_n^\F(\epsilon)$.

\item Define $\In(\epsilon) := \vartheta(1; \epsilon)$, the inertia of
every matrix in $LPM_n^\F(\epsilon)$. Then $\In(\epsrev) =
\In(\epsilon)$.
\end{enumerate}
Analogous statements hold for the cone $TPM_n^\F$.
\end{theorem}

\begin{proof}\hfill
\begin{enumerate}[(1)]
\item This is shown using a ``non-orthogonal spectral theorem'', which is
a consequence of Sylvester's law of inertia (over $\C$). This result was
stated only over the reals in~\cite{BGKP-inertia}, but is valid over $\C$
too:

\begin{lemma}[{\cite[Lemma~2.3]{BGKP-inertia}}]\label{Linertia}
Let $0\leq s,t\leq s+t\leq n$ be integers. If ${\bf u}_1,\dots,{\bf
u}_s,{\bf v}_1,\dots,{\bf v}_t\in \C^n$ are linearly independent, then
the Hermitian matrix ${\bf u}_1{\bf u}_1^*+\dots+{\bf u}_s{\bf
u}_s^*-{\bf v}_1{\bf v}_1^*-\dots-{\bf v}_t{\bf v}_t^*$ has exactly $s$
positive and $t$ negative eigenvalues.  
\end{lemma}

Now since $\F \subseteq \C$, every matrix in $LPM_n^\F(\epsilon)
\subseteq LPM_n^\C(\epsilon)$ has a Cholesky decomposition $A = L
\D_\epsilon L^*$ over $\C$. Writing $L = [ {\bf c}_1 | \cdots | {\bf c}_n
]$ in column form, $A = L \D_\epsilon L^T = \sum_{j=1}^n
(\D_\epsilon)_{jj} {\bf c}_j {\bf c}_j^*$, and so by Lemma~\ref{Linertia}
it has exactly $k$ negative eigenvalues, where $k$ is the number of
negative entries in $\D_\epsilon$. One checks this is precisely
$\vartheta(1;\epsilon)$, since if we set $\epsilon_0 := 1$, then an entry
$(\D_{\epsilon})_{jj} = \epsilon_{j-1}\epsilon_j$ is negative if and only
if $\epsilon_{j-1} \to \epsilon_j$ changes sign.

\item Immediate, since $\epsilon \in \{ \pm 1 \}^n \ \mapsto \
\D_\epsilon$ is a bijection.

\item This holds via~\eqref{Esanity}, because $\D_\epsilon =
\D_\epsilon^{-1}$ and $P_n \D_\epsilon^{-1} P_n$ have the same number of
$-1$ entries. \qedhere
\end{enumerate}
\end{proof}

We end by defining probability distributions supported on the inertial
cones $\In_n^\R(k) \subset LPM_n^\R$. This allows us to sample from them,
which generalizes sampling from the cone $\In_n^\R(0) = PD_n$:

\begin{prop}\label{Pcloning}
Fix integers $n \geq 1$ and $0 \leq k \leq n$, and let $Q$ be as in
Theorem~\ref{Twishart}.
\begin{enumerate}[$(1)$]
\item If $Q$ is supported on $PD_n$, then the following ``cloning'' of
$Q$ is supported on $\In^\R_n(k)$: the union of the transferred
distributions (as in Definition~\ref{Dwishart} or Theorem~\ref{Twishart})
\begin{equation}
\bigsqcup_{\epsilon \, : \, \vartheta(1;\epsilon) = k}
\frac{1}{\binom{n}{k}} Q^{LPM}_\epsilon :
\bigsqcup_{\vartheta(1;\epsilon)=k} LPM^\R_n(\epsilon) \to [0,\infty).
\end{equation}

\item One can replace the denominator $\binom{n}{k}$ by $2^n$ and take
the union over all $\epsilon \in \{ \pm 1 \}^n$, to obtain a distribution
supported on the open dense cone $LPM_n^\R$, which clones the
distribution $Q$ on $PD_n$.

\item The analogous statements for the TPM cones also hold true.
\end{enumerate}
\end{prop}

\subsection{Further questions}

In addition to the broader goals of developing in detail the
numerical, geometric, probabilistic, statistical, and applied
ramifications of our results above, we end with a few specific
mathematical questions that naturally emerge from this work. The first
question has to do with Example~\ref{Exnotconvex}, which showed that all
cones $LPM_n(\epsilon)$ are not convex if $\epsilon_2 = -1$.

\begin{question}
For which $n \geq 1$ and sign patterns $\epsilon$ is the cone
$LPM_n(\epsilon)$ -- or $TPM_n(\epsrev)$ -- convex?
\end{question}

Another question is as follows.

\begin{question}
Can one extend these Cholesky-type factorizations in a ``consistent''
manner to the entire real symmetric / Hermitian cone -- which subsumes
the usual Cholesky decomposition of the positive matrices
$\overline{PD_n}$?
\end{question}

Our final concrete question is a hands-on one about Wishart densities on
LPM/TPM cones.

\begin{question}
What are the moments of the Wishart density
$W_{\epsilon,n}^{LPM}(\Sigma,N)$ -- or equivalently (by
Theorem~\ref{Twishart}), of $W_{\epsilon,n}^{TPM}(\rev{\Sigma},N)$?
How about the moments of the other densities introduced above?
\end{question}

\subsection{Acknowledgments}

We thank Manjunath Krishnapur for useful discussions re: the Bartlett
decomposition and Remark~\ref{Rnatural}.
A.K.\ was partially supported by a Shanti Swarup Bhatnagar Award from
CSIR (Govt.\ of India). P.K.V.\ was supported by a Centre de recherches
math\'ematiques and Universit\'e Laval (CRM-Laval) Postdoctoral
Fellowship and the Alliance grant.



\appendix
\section{Direct sums and tensor products of LPM cones}\label{Soperations}

In this section, we describe additional operations on the cones
$LPM_n^\F, TPM_n^\F$. We have already seen the reversal map $\epsilon
\mapsto \epsrev$ in the context of~\eqref{Ediffeos}, and $\epsilon \circ
\epsilon'$ in the context of the $\boxdot$ operation on $LPM_n^\F,
TPM_n^\F$. Here is another operation, this time relating sign patterns of
different lengths.

\begin{prop}\label{Poplus}
Given sign patterns $\epsilon \in \{ \pm 1 \}^n, \epsilon' \in \{ \pm 1
\}^{n'}$ for $n,n' \geq 1$, define their {\em direct sum} 
(which is associative but not commutative)
\begin{equation}
\epsilon \oplus \epsilon' := (\epsilon_1, \dots, \epsilon_n; \,
\epsilon_n \epsilon'_1, \dots, \epsilon_n \epsilon'_{n'}) \in \{ \pm 1
\}^{n+n'}.
\end{equation}
Also define the {\em direct sum} of matrices $A,A'$ to be $A
\oplus A' := \begin{pmatrix} A & {\bf 0} \\ {\bf 0}^T & A'
\end{pmatrix}$; and fix a subfield $\F \subseteq \C$.
\begin{enumerate}[$(1)$]
\item Then $\bigsqcup_{n \geq 1} {\bf L}_n^\F$ is a nonabelian semigroup
under $\oplus$, with the operations $J \oplus -$ and $- \oplus J$ both
isometries $: {\bf L}_n^\F \to {\bf L}_{n+n'}^\F$ for all $n,n' \geq 1$
and $J \in {\bf L}_{n'}^\F$, under the
log-Cholesky metrics~\eqref{ElogCholesky}:
\begin{equation}
d_{{\bf L}^\F_{n+n'}}(J \oplus L, J \oplus K) = d_{{\bf L}_n^\F}(L,K) =
d_{{\bf L}^\F_{n+n'}}(L \oplus J, K \oplus J),
\qquad \forall L,K \in {\bf L}_n^\F.
\end{equation}
\item The operation $\oplus$ is compatible with the
Cholesky-decomposition maps $\Phi$. Namely,
\[
\D_{\epsilon \oplus \epsilon'} = \D_\epsilon \oplus \D_{\epsilon'},
\qquad \text{and} \qquad B_\epsilon \oplus B_{\epsilon'} \in
LPM_n(\epsilon \oplus \epsilon')
\ \forall B_\epsilon \in LPM^\F_n(\epsilon), B_{\epsilon'} \in
LPM^\F_{n'}(\epsilon').
\]
Moreover, $\Phi_{B_\epsilon \oplus B_{\epsilon'}}(L \oplus L') =
\Phi_{B_\epsilon}(L) \oplus \Phi_{B_{\epsilon'}}(L')$ for all $L \in
{\bf L}_n, L' \in {\bf L}_{n'}$.

\item The set
$\displaystyle \widetilde{LPM}_{\boldsymbol \sqcup}^\F := \bigsqcup_{n
\geq 1} \prod_{\epsilon \in \{ \pm 1 \}^n} LPM_n^\F(\epsilon)$ and
is a nonabelian semigroup under $\oplus$, and
\[
\widetilde{\Phi}_\D := \bigsqcup_{n \geq 1}
(\Phi_{\D_\epsilon})_{\epsilon \in \{ \pm 1 \}^n} \ : \ \bigsqcup_{n \geq
1} {\bf L}_n^\F \ \to \ \widetilde{LPM}_{\boldsymbol \sqcup}^\F
\]
is a semigroup monomorphism.
(Thus, $\left. \widetilde{\Phi}_\D \right|_{{\bf L}_n^{\F}} =
(\Phi_{\D_\epsilon})_{\epsilon \in \{ \pm 1 \}^n}$ for each $n$.)

\item The following sets are also semigroups under $\oplus$:
\begin{equation}\label{Esqcup}
\pm PD_{\boldsymbol \sqcup}^\F := \bigsqcup_{n \geq 1} \pm PD_n^\F \
\subset \ LPM_{\boldsymbol \sqcup}^\F := \bigsqcup_{n \geq 1}
\bigsqcup_{\epsilon \in \{ \pm 1 \}^n} LPM_n^\F(\epsilon).
\end{equation}
The maps $\bigsqcup_{n \geq 1} \Phi_{\pm \Id_n} : \bigsqcup_{n \geq 1}
{\bf L}_n^\F \to \pm PD_{\boldsymbol \sqcup}^\F$, sending $L \mapsto \pm
L L^*$ are semigroup morphisms.

More generally, fix $m \geq 1$ and a sign pattern $\epsilon \in \{ \pm 1
\}^m$. Then the sets
$\bigsqcup_{n \geq 1} {\bf L}_{nm}^\F$ and $\bigsqcup_{n \geq 1}
LPM_{nm}^\F(\epsilon^{\oplus n})$
are nonabelian semigroups under $\oplus$, and the map
$\bigsqcup_{n \geq 1} \Phi_{\D_{\epsilon^{\oplus n}} }$
between them is a semigroup monomorphism, which is onto if $\F \cap
(0,\infty)$ is closed under positive square roots.
\end{enumerate}
\end{prop}

\begin{proof}
We only explain (3)~the assertions involving
$\widetilde{LPM}^\F_{\boldsymbol \sqcup}$. First, the operation on
$\widetilde{LPM}^\F_{\boldsymbol \sqcup}$ is:
\[
(A_\epsilon)_{\epsilon \in \{ \pm 1 \}^n} \oplus
(B_{\epsilon'})_{\epsilon' \in \{ \pm 1 \}^{n'}} :=
(A_\epsilon \oplus B_{\epsilon'})_{(\epsilon, \epsilon') \in \{ \pm 1
\}^{n+n'}},
\]
where we index the terms on the right using that the map $: (\epsilon,
\epsilon') \mapsto \epsilon \oplus \epsilon'$ is one-to-one, hence a
bijection of $\{ \pm 1 \}^{n+n'}$. Next, 
$\widetilde{\Phi}_{\D}$ is injective by Theorem~\ref{Ttower};
and given $L \in {\bf L}_n$ and $L' \in {\bf L}_{n'}$ for some $n,n'
\geq 1$ (which may be equal), we compute using part~(2) and the
above bijection:
\[
\bigsqcup_{\epsilon \in \{ \pm 1 \}^n} \Phi_{\D_\epsilon}(L) \oplus
\bigsqcup_{\epsilon' \in \{ \pm 1 \}^{n'}} \Phi_{\D_{\epsilon'}}(L') =
\bigsqcup_{\epsilon'' = \epsilon \oplus \epsilon' \in \{ \pm 1 \}^{n+n'}}
\Phi_{\D_\epsilon}(L) \oplus \Phi_{\D_{\epsilon'}}(L') =
\bigsqcup_{\epsilon''} \Phi_{\D_{\epsilon''}}(L \oplus L'). \qedhere
\]
\end{proof}

We collect together additional facts into two results: the first for the
metric groups $(LPM_n^\F(\epsilon), \circledast)$ for all $n,\epsilon$,
and the second for their unions $(LPM_n^\F, \boxdot)$.

\begin{theorem}\label{Toplus1}
Fix $\F \subseteq \C$ with $\F \cap (0,\infty)$ closed
under positive square roots. Let $B_\epsilon := \D_\epsilon\ \forall
\epsilon$.
\begin{enumerate}[$(1)$]
\item For all $\epsilon' \in \{ \pm 1 \}^{n'}$ and $J \in {\bf L}_{n'}$,
the following maps are isometries under $\circledast$:
\begin{align*}
- \oplus J \D_{\epsilon'} J^* : &\ LPM_n^\F(\epsilon) \to
LPM_{n+n'}^\F(\epsilon \oplus \epsilon'),\\
J \D_{\epsilon'} J^* \oplus - : &\ LPM_n^\F(\epsilon) \to
LPM_{n'+n}^\F(\epsilon' \oplus \epsilon).
\end{align*}

\item The operation $\oplus$ is a(n additive and not bi-additive) map of
2-divisible abelian groups $: LPM_n^\F(\epsilon) \times
LPM_{n'}^\F(\epsilon') \hookrightarrow LPM_{n+n'}^\F(\epsilon \oplus
\epsilon')$, all under $\circledast$.

\item If moreover $\F = \R$ or $\C$, then $\oplus : LPM_n^\F(\epsilon)
\times LPM_{n'}^\F(\epsilon') \hookrightarrow LPM_{n+n'}^\F(\epsilon
\oplus \epsilon')$ is in fact an $\R$-linear isometric monomorphism of
real Euclidean spaces, which translates to the embedding $:
(\R^{n(n+1)/2} \times \R^{n'(n'+1)/2}, \| \cdot \|_2) \hookrightarrow
(\R^{(n+n')(n+n'+1)/2}, \| \cdot \|_2)$ or $\R^{n^2} \times \R^{(n')^2}
\hookrightarrow \R^{(n+n')^2}$. (This also shows that the inclusions here
and in part~(3) are $\R$-linear, not $\R$-bilinear.)
\end{enumerate}
The analogous statements for TPM matrices also hold true.
\end{theorem}

\begin{theorem}\label{Toplus2}
Setting as in Theorem~\ref{Toplus1}.
\begin{enumerate}[$(1)$]
\item The group operation $\boxdot$ is also compatible with $\oplus$: the
operation $\oplus$ is an additive map of abelian groups $: LPM_n^\F
\times LPM_{n'}^\F \hookrightarrow LPM_{n+n'}^\F$, all under $\boxdot$.

\item Theorem~\ref{Toplus1}(1) extends to all of $LPM_n^\F$: for all
$\epsilon' \in \{ \pm 1 \}^{n'}$, $J \in {\bf L}_{n'}$, and $p \in
[1,\infty]$, the maps
$- \oplus J \D_{\epsilon'} J^*$ and $J \D_{\epsilon'} J^* \oplus - \quad
: LPM_n^\F \to LPM_{n+n'}^\F$
are isometries under $d_p$~\eqref{Edp}.

\item The analogous statements for TPM matrices also hold true. Moreover,
$\displaystyle \overset{\longleftarrow}{\epsilon \oplus \epsilon'} =
{\overset{\leftarrow}{\epsilon'}} \oplus \epsrev$;
and if $n=n'$, then
$\displaystyle \overset{\longleftarrow}{\epsilon \circ \epsilon'} =
\epsrev \circ \overset{\leftarrow}{\epsilon'}$ (see
Definition~\ref{Dschurprod}). More generally, $\overset{\longleftarrow}{A
\oplus B} = \rev{B} \oplus \rev{A}$.
\end{enumerate}
\end{theorem}

\begin{proof}
The proofs of all but the final part here (and all of
Theorem~\ref{Toplus1}) are omitted, as they reveal no surprises. 
Next, the second equation in the final part emerges out of the explicit
computations used in proving part~(1); and the final assertion is
straightforward.

To show the first equation in the final part, by part~(2) it
suffices to consider one matrix each from $LPM_n(\epsilon),
LPM_{n'}(\epsilon')$; we choose $\D_\epsilon, \D_{\epsilon'}$,
respectively. Now the first assertion follows via~\eqref{Esanity}:
\[
\D_{\overset{\longleftarrow}{\epsilon \oplus \epsilon'}} =
P_{n+n'} \D_{\epsilon \oplus \epsilon'}^{-1} P_{n+n'} = P_{n+n'}
(\D_\epsilon \oplus \D_{\epsilon'}) P_{n+n'} = (P_{n'} \D_{\epsilon'}
P_{n'}) \oplus (P_n \D_\epsilon P_n) =
\D_{\overset{\leftarrow}{\epsilon'}} \oplus \D_{\epsrev}. \qedhere
\]
\end{proof}

\begin{remark}
Proposition~\ref{Poplus} and Theorems~\ref{Toplus1} and~\ref{Toplus2}
have ``real-closed'' variants (except for the assertions about metrics
and Euclidean spaces). Here one works with $\F$ as in
Theorem~\ref{Trealclosed}, and the proofs are verbatim as those above.
\qed
\end{remark}

We next come to the tensor/Kronecker product of matrices:
$A \otimes B := \begin{pmatrix} a_{11} B & a_{12} B & \cdots \\ a_{21} B
& a_{22} B & \cdots \\ \vdots & \vdots & \ddots \end{pmatrix}$. This
operation enjoys several compatibility properties: it is associative,
behaves well with transposes ($(A \otimes B)^T = A^T \otimes B^T$) and
with conjugation, hence with the complex-adjoint. The tensor product of
upper/lower triangular matrices has the same property; it is also
multiplicative: $(A \otimes C)(B \otimes D) = AB \otimes CD$ when all
terms are defined; and it distributes over addition in either factor.

However, $\otimes$ does not behave well (distribute) with $\oplus$. E.g.\
say $A = \begin{pmatrix} a & b \\ c & d \end{pmatrix} \in \R^{2 \times
2}$. Then
\[
A \otimes (B \oplus B') = \begin{pmatrix}
aB & 0 & bB & 0 \\
0 & aB' & 0 & bB' \\
cB & 0 & dB & 0 \\
0 & cB' & 0 & dB'
\end{pmatrix}, \qquad
(A \otimes B) \oplus (A \otimes B') = \begin{pmatrix}
aB & bB & 0 & 0\\
cB & dB & 0 & 0\\
0 & 0 & aB' & bB'\\
0 & 0 & cB' & dB'
\end{pmatrix}.
\]
These matrices are permutationally similar/conjugate, but are not always
equal -- which is required when considering leading principal minors.

Similarly, it is not universally true that $(A \circledast A') \otimes B
\neq (A \otimes B) \circledast (A' \otimes B)$. To see this, suppose $A =
L \D_\epsilon L^T$, $A' = L' \D_\epsilon (L')^T$, $B = K \D_\delta K^T$.
Then the left-hand side equals
\begin{align*}
(A \circledast A') \otimes B = &\ (L \circledcirc L') \D_\epsilon (L
\circledcirc L')^T \otimes (K \D_\delta K^T)\\
= &\ ((L \circledcirc L') \otimes K) (\D_\epsilon \otimes \D_\delta)
((L \circledcirc L') \otimes K)^T,
\end{align*}
while the right-hand side equals
\begin{align*}
(A \otimes B) \circledast (A' \otimes B) = &\ [(L \otimes K) \circledcirc
(L' \otimes K)] (\D_\epsilon \otimes \D_\delta) [(L^T \otimes K^T)
\circledcirc ((L')^T \otimes K^T)].
\end{align*}
By Theorem~\ref{Tlpm}, we need to check if the Cholesky factors agree:
\[
(L \circledcirc L') \otimes K \equiv (L \otimes K) \circledcirc (L'
\otimes K),
\]
and this is not universally true.

Other properties that do not hold universally -- on the Cholesky level --
are:
\begin{align}
\begin{aligned}
L \otimes (K \circledcirc K') &\ \not\equiv (L \otimes K) \circledcirc
(L \otimes K'),\\
(L \circledcirc L') \otimes (K \circledcirc K') &\ \not\equiv (L \otimes
K) \circledcirc (L' \otimes K'),\\
d_{{\bf L}_{nn'}}(J \otimes L, J \otimes K) &\ \not\equiv d_{{\bf
L}_n}(L,K),
\end{aligned}
\end{align}
and so they do not hold on the level of LPM spaces either. However, a few
properties do hold -- for instance, it may not be immediately obvious
that if $A_\epsilon \in LPM_n(\epsilon)$ and $A_{\epsilon'} \in
LPM_{n'}(\epsilon')$ then $A_\epsilon \otimes A_{\epsilon'} \in
LPM_{nn'}(\epsilon \otimes \epsilon')$. We show this now, not just over
$\R$ but more generally.

\begin{theorem}
Let $\mathbb{E}$ be a real-closed field and $\F \subseteq
\overline{\mathbb{E}} = \mathbb{E}[\sqrt{-1}]$.
Fix integers $n,n' \geq 1$ and sign patterns $\epsilon, \delta \in \{ \pm
1 \}^n$ and $\epsilon', \delta' \in \{ \pm 1 \}^{n'}$. 
Define the {\em tensor product} of $\epsilon, \epsilon'$ via the equation
\begin{equation}
\D_{\epsilon \otimes \epsilon'} := \D_\epsilon \otimes \D_{\epsilon'}.
\end{equation}
\begin{enumerate}[$(1)$]
\item The set $\bigsqcup_{n \geq 1} {\bf L}^\F_n$ is a nonabelian monoid
under $\otimes$, with identity $(1)_{1 \times 1}$.

\item The tensor product behaves well with reversal and with
componentwise products:
\begin{equation}\label{Etensorrev}
\overset{\longleftarrow}{\epsilon \otimes \epsilon'} = \epsrev \otimes
\overset{\leftarrow}{\epsilon'}, \qquad
(\epsilon \otimes \epsilon') \circ (\delta \otimes \delta') = (\epsilon
\circ \delta) \otimes (\epsilon' \circ \delta').
\end{equation}
More strongly, $\overset{\longleftarrow}{A \otimes B} = \rev{A} \otimes
\rev{B}$ for all square matrices $A_{n \times n}, B_{n' \times n'}$.

\item The operation $\otimes$ is compatible with the Cholesky
decomposition: with $\F$ as in Theorem~\ref{Trealclosed},
\begin{equation}\label{Etensor}
\Phi_{\D_{\epsilon \otimes \epsilon'}}(L \otimes L') =
\Phi_{\D_\epsilon}(L) \otimes \Phi_{\D_{\epsilon'}}(L'), \qquad \forall L
\in {\bf L}^\F_n, L' \in {\bf L}^\F_{n'}.
\end{equation}
More generally, $\Phi_{B_\epsilon \otimes B_{\epsilon'}}(L \otimes L') =
\Phi_{B_\epsilon}(L) \otimes \Phi_{B_{\epsilon'}}(L')$ for all $L \in
{\bf L}_n^\F, L' \in {\bf L}_{n'}^\F$ and $B_\epsilon \in
LPM_n^\F(\epsilon), B_{\epsilon'} \in LPM_{n'}^\F(\epsilon')$.

\item The sets $+PD_{\boldsymbol \sqcup}^\F \subset LPM^\F_{\boldsymbol
\sqcup}$ defined in~\eqref{Esqcup} are nonabelian monoids under
$\otimes$, with identity $(1)_{1 \times 1}$. More precisely,
for all $n,n' \geq 1$, $\epsilon \in \{ \pm 1 \}^n$, and $\epsilon' \in
\{ \pm 1 \}^{n'}$, we have
\begin{equation}\label{Elpmtensor}
LPM_n^\F(\epsilon) \otimes LPM_{n'}^\F(\epsilon') \subseteq
LPM_{nn'}^\F(\epsilon \otimes \epsilon').
\end{equation}
Moreover, $\Phi_+ := \bigsqcup_{n \geq 1} \Phi_{\Id_n}$ is a monoid
monomorphism $: \bigsqcup_n {\bf L}_n^\F \to +PD_{\boldsymbol
\sqcup}^\F$, which is surjective if $\F \cap (0,\infty)$ is closed under
positive square roots.
\end{enumerate}
Similar assertions hold for tensor products of TPM matrices.
\end{theorem}

\begin{proof}
The first part is easy. Next, computing the sign patterns on both
sides of the equations
\begin{align}
P_{nn'} \D_{\epsilon \otimes \epsilon'} P_{nn'} = &\ (P_n \otimes P_{n'})
(\D_\epsilon \otimes \D_{\epsilon'}) (P_n \otimes P_{n'}) = P_n
\D_\epsilon P_n \otimes P_{n'} \D_{\epsilon'} P_{n'},\\
(\D_\epsilon \otimes \D_{\epsilon'})(\D_\delta \otimes \D_{\delta'}) = &\
\D_\epsilon \D_\delta \otimes \D_{\epsilon'} \D_{\delta'}\notag
\end{align}
yields~\eqref{Etensorrev}. The next line and the third part follow from
the tensor product and the reversal map being compatible with usual
multiplication, transposes, etc. We now show the final part.
The first assertion about $+PD_{\boldsymbol \sqcup}^\F$ follows
from~\eqref{Elpmtensor}. To show~\eqref{Elpmtensor}, write $A_\epsilon =
L \D_\epsilon L^*$ and $A_{\epsilon'} = L' \D_{\epsilon'} (L')^*$ via
Theorem~\ref{Trealclosed} over $\F \leadsto \mathbb{E}[\sqrt{-1}]$;
now~\eqref{Elpmtensor} follows from~\eqref{Etensor}. Finally, the claims
about $\Phi_+$ follow from Theorems~\ref{Ttower} and~\ref{Tlpm} (or their
real-closed analogues in Theorem~\ref{Trealclosed}).
\end{proof}

\begin{remark}
For completeness, we record that the direct sum and inertia interact via:
\begin{equation}
\In(\epsilon \oplus \epsilon') = \In(\epsilon) + \In(\epsilon') + {\bf
1}_{\epsilon_{n-1} \neq \epsilon_n} \epsilon'_1.
\end{equation}
This is an explicit calculation, with two cases for $\epsilon_{n-1}
\epsilon_n = \pm 1$.  It also suggests a similar ``explicit'' -- but
perhaps cumbersome -- formula for $\In(\epsilon \otimes \epsilon')$,
which we do not pursue here.
\end{remark}

\section{SSRPM matrices}\label{Sssrpm}

Here we return to Section~\ref{Sublpm}, where we listed three possible
approaches to generalize the notion of positive definiteness to other
sign patterns.
The second was via leading principal minors, and led to the Cholesky
factorization-diffeomorphisms for $LPM_n(\epsilon)$ matrices.
The third led to $TPM_n(\epsilon)$ matrices. We now explore the first
option listed there.

\begin{defn}
Given $n \geq 1$ and a sign pattern $\epsilon \in \{
\pm 1 \}^n$, a real symmetric matrix $A_{n \times n}$ is said to be
$SSRPM_n(\epsilon)$ (\textit{Strictly Sign-Regular Principal Minors} with
pattern $\epsilon$) if every principal $k \times k$ minor of $A$ is
nonzero with sign $\epsilon_k$. We say $A_{n \times n}$ is \textit{SSRPM}
if $A$ is $SSRPM_n(\epsilon)$ for some $\epsilon$.
\end{defn}

Clearly, $SSRPM_n(\epsilon) \subseteq LPM_n(\epsilon)$ for all $n$ and
$\epsilon$.
At the same time, $SSRPM_n({\bf 1}_n) = PD_n = LPM_n({\bf 1}_n) =
TPM_n({\bf 1}_n)$ by the classical Sylvester criterion. This immediately
gives the same result for negative definite matrices:
\[
SSRPM_n(\epsilon^-_n) = -PD_n = LPM_n(\epsilon^-_n) \quad \text{for}
\quad \epsilon^-_n := (-1, (-1)^2, \dots, (-1)^n).
\]

Akin to both of these, it is natural to ask if for any other sign pattern
$\epsilon$, the leading principal minors determine the signs of all
principal minors. As we now explain, this does not happen for any other
sign pattern; in doing so, we also identify all diagonal $SSRPM_n$
matrices.

\begin{lemma}\label{Llpm}
Fix an integer $n \geq 1$ and a sign pattern $\epsilon \in \{ \pm 1
\}^n$. Then $SSRPM_n(\epsilon) \subseteq LPM_n(\epsilon)$. Moreover, the
following are equivalent:
\begin{enumerate}
\item[$(1)$] Equality holds: $SSRPM_n(\epsilon) = LPM_n(\epsilon)$.

\item[$(2)$] $SSRPM_n(\epsilon)$ contains a diagonal matrix.

\item[$(3)$] $\epsilon = {\bf 1}_n$ or $\epsilon^-_n$ (defined above).
\end{enumerate}
Via~\eqref{Ediffeos}, the same statement holds for $TPM_n(\epsilon)$.
\end{lemma}

\begin{proof}
The inclusion is obvious. Next, for $\epsilon = {\bf 1}_n$, we have
$LP_n({\bf 1}_n) = PD_n = SSRPM_n({\bf 1}_n)$, by Sylvester's criterion.
By taking negatives, we get the desired equality for $\epsilon_n^- = (-1,
(-1)^2, \dots, (-1)^n)$. Thus $(3) \implies (1)$. Conversely, suppose
$\epsilon$ is any other sign pattern. Then the diagonal matrix with
$(1,1)$ entry $\epsilon_1$ and $(k,k)$ entry $\epsilon_k /
\epsilon_{k-1}$ for $k>1$, belongs to $LPM_n(\epsilon)$ -- in particular,
this set is never empty. But by choice, this diagonal matrix has both $1$
and $-1$ as diagonal entries. Thus it is not in $SSRPM_n$, proving $(1)
\implies (3)$.

Finally, if $D \in SSRPM_n(\epsilon)$ is diagonal then all diagonal
entries have the same sign $+$ or $-$, in which case $\epsilon = {\bf
1}_n$ or $\epsilon^-_n$. Conversely, if $\pm D \in (0,\infty)^n$ then it
is clear that $D \in SSRPM_n(\epsilon)$ for $\epsilon = {\bf 1}_n$ or
$\epsilon^-_n$. Thus $(2) \Longleftrightarrow (3)$. The proof for $TPM_n$
matrices indeed follows via $A \mapsto P_n A P_n$.
\end{proof}

\begin{remark}\label{Rssrpm}
Let $\epsilon \in \{ \pm 1 \}^n \setminus \{ {\bf 1}_n, \epsilon^-_n \}$.
In light of the inclusion $SSRPM_n(\epsilon) \subsetneq LPM_n(\epsilon)$,
we now explain why a variant of Theorem~\ref{Tlpm} does not hold for
SSRPM matrices. Indeed, both directions of the earlier bijection $L
\longleftrightarrow L B_\epsilon L^T$ do not exactly work out. Namely:
let $B_\epsilon \in SSRPM_n(\epsilon)$. Hence $B_\epsilon \in
LPM_n(\epsilon)$, so by Theorem~\ref{Tlpm}, $\{ L B_\epsilon L^T : L \in
{\bf L}_n \}$ equals all of $LPM_n(\epsilon)$ and not merely the proper
subset $SSRPM_n(\epsilon)$. Viewed ``dually'', indeed every $A \in
SSRPM_n(\epsilon)$ admits a Cholesky decomposition, but so do (other)
matrices $A \in LPM_n(\epsilon) \setminus SSRPM_n(\epsilon)$.
\end{remark}

Given Remark~\ref{Rssrpm}, we finish by discussing a basic question:
\textit{do such matrices exist?} More precisely, is the set
$SSRPM_n(\epsilon)$ nonempty for all $n \geq 1$ and all $\epsilon \in \{
\pm 1 \}^n$? As we now show via examples, the answer is indeed ``yes''
for $n \leq 3$; and for general $n$ we provide $2n+4$ distinct $\epsilon$
for which such matrices exist. We are unaware of other $\epsilon$ (e.g.\
for $n \geq 6$) for which such matrices exist.

The first class of $(2n+2)$-many examples of $\epsilon$ are ``nicer''
than is required: not only are all $k \times k$ principal minors for each
fixed $k$ of the same sign, but they have the same value, and even
stronger:
the submatrices are all the same. (This restrictiveness means that we
obtain only $2n+2$ out of the $2^n$ possible sign patterns.)

\begin{example}\label{Exisotropic}
We show that the following $2n+2$ sign patterns are realizable for all
$n$:
\[
\mathcal{E}^{(k)} := ({\bf 1}_k; - {\bf 1}_{n-k}), \quad
\mathcal{E}^{(k)} \circ (-1,(-1)^2,\dots,(-1)^n), \qquad 0 \leq k \leq n.
\]

To realize these, let $a,b \in \R$ and let
\begin{equation}\label{EMabn}
M(a,b,n) := b {\bf 1}_{n \times n} + (a-b) \Id_n = \begin{pmatrix}
a & b & \cdots & b \\
b & a & \cdots & b \\
\vdots & \vdots & \ddots & \vdots \\
b & b & \cdots & a
\end{pmatrix}.
\end{equation}
This is a real symmetric $n \times n$ Toeplitz matrix with diagonal
entries $a$ and off-diagonal entries $b$. As the only nonzero eigenvalue
of the rank-one $b {\bf 1}_{n \times n}$ is its trace $nb$, hence all but
one eigenvalues of $M(a,b,n)$ are $(a-b)$, and the remaining one is $a +
(n-1)b$. As any principal $k \times k$ submatrix of $M(a,b,n)$ is
$M(a,b,k)$, all $k \times k$ principal minors of $M(a,b,n)$ are
$(a+(k-1)b)(a-b)^{k-1}$.

As we want all minors nonzero, we have $a \neq 0$; and either $a>b$ or
$a<b$. We consider three sub-cases for each; in all of them, $M(a,b,n)
\in SSRPM_n(\epsilon)$, and we write down the sign pattern $\epsilon \in
\{ \pm 1 \}^n$. First for the cases when $a>b$:
\begin{enumerate}[(1)]
\item If $a > b \geq 0$, then $\epsilon = {\bf 1}_n$.
(This sign pattern also occurs in the next case.)

\item If $a > 0 > b$, then based on the value of $b/a$, $\epsilon = ({\bf
1}_k; -{\bf 1}_{n-k})$ is possible for any $1 \leq k \leq n$.

\item If $0 > a > b$, then $\epsilon = -{\bf 1}_n$.
\end{enumerate}

Now for the three cases where $a<b$:
\begin{enumerate}[(1)]
\setcounter{enumi}{3}
\item First if $a < b \leq 0$, then $\epsilon = (-1, 1, \dots, (-1)^n)$.
(This sign pattern also occurs in the next case.)

\item If $a < 0 \leq b$, then  we see that the sign pattern $\epsilon =
(-1, 1, \dots, (-1)^k; \ (-1)^k, (-1)^{k+1}, \dots, (-1)^{n-1})$ is
possible for any $1 \leq k \leq n$, depending on the value of $b/a$.

\item Finally, 
if $0 < a < b$, then $\epsilon = (1, -1, \dots, (-1)^{n-1})$.
\end{enumerate}
Notice that the last three cases precisely correspond to the first three,
under $A \mapsto -A$. \qed
\end{example}

\begin{remark}\label{Rexamples}
The real symmetric matrices discussed and classified above include
positive and negative definite matrices, but also:
\textit{symmetric N-matrices} (ones whose all principal minors are
negative) \cite{Inada},
\textit{symmetric almost P-matrices} ($\det(A) < 0$ but all proper
principal minors are positive),
and \textit{symmetric PN-matrices} (ones whose $k \times k$ minors have
sign $(-1)^{k-1}$) \cite{Maybee}.
These matrix classes (including their non-symmetric counterparts) have
been widely studied in a multitude of theoretical and applied fields, see
the remarks after Definition~\ref{Dlpm}.
\end{remark}

\begin{example}
Another cone of structured (real symmetric) matrices is that of the
\textit{(symmetric) almost N-matrices} \cite{Descartes} -- these have
positive determinant but all proper principal minors negative. We
construct a 3-parameter family of such matrices for each $n \geq 3$; for
$n=2$ it is easy to construct examples.
Thus, let $n \geq 3$, and choose scalars $a,b,c$ satisfying:
\begin{equation}\label{Econstraints}
0 > a > b, \qquad \frac{(n-2)b^2}{a+(n-3)b} < c <
\frac{(n-1)b^2}{a+(n-2)b} < 0.
\end{equation}
One checks that the final inequalities for $c$ are consistent; and
moreover,
\[
(n-1)b^2 < ac  +(n-2)bc < (n-1)bc \quad \implies \quad c<b.
\]

Now we present the matrices in question. Let $M(a,b,c,n) \in \R^{n \times
n}$ have all off-diagonal entries $b$, the $(n,n)$ entry $c$, and all
other diagonal entries $a$, i.e.,
\[
M(a,b,c,n) = M(a,b,n) + (c-a) E_{nn},
\]
where $M(a,b,n)$ is as in~\eqref{EMabn}.

We claim that $M(a,b,c,n)$ is a symmetric
almost N-matrix, i.e.\ in $SSRPM_n((-1,\dots,-1,1))$ if $a,b,c$
satisfy~\eqref{Econstraints}.
To show the claim, first note that all diagonal entries are negative from
above. Next, any principal $k \times k$ minor of the leading principal
$(n-1) \times (n-1)$ submatrix $M(a,b,n-1)$ equals $(a-b)^{k-1} (a +
(k-1)b) < 0$ (see Example~\ref{Exisotropic}).

All other principal minors contain the $(n,n)$ entry $c$, and hence are
of the form $M(a,b,c,k)$. Thus, it suffices to show that
\[
\det M(a,b,c,k) < 0 < \det M(a,b,c,n), \quad \forall 2 \leq k \leq n-1.
\]

We now compute $\det M(a,b,c,k)$ by expanding along the first column and
linearity in it. Since the first column equals $(a,b,\dots,b)^T + (c-a)
{\bf e}_1$, we get:
\begin{align*}
\det M(a,b,c,k) = &\ \det M(a,b,k) + (c-a) \det M(a,b,k-1)\\
= &\ (a-b)^{k-2} \left[ (a-b)(a+(k-1)b) + (c-a)(a+(k-2)b) \right]\\
= &\ (a-b)^{k-2} \left[ ca-b^2 + (k-2)b(c-b) \right].
\end{align*}

To determine the signs of these determinants, since $a>b$ we may consider
$f(2), \dots, f(n-1), f(n)$, where $f(k) := ca-b^2 + (k-2)b(c-b)$. But
$f(k)$ is an arithmetic progression with common step-size $b(c-b)$, which
is positive by~\eqref{Econstraints} and the following computation.
Thus, it suffices to verify that $f(n-1) < 0 < f(n)$, i.e.,
\[
ca-b^2 + (n-3)b(c-b) < 0 < ca-b^2 + (n-2)b(c-b).
\]
Rearranging and solving for $c$, this is precisely equivalent to the
bounds on $c$ in~\eqref{Econstraints}. \qed
\end{example}

Our final, $(2n+4)$th example is the ``negative'' of the preceding matrix
cone:

\begin{example}
Let $A$ be any symmetric almost $N$-matrix. Then
\[
-A \in SSRPM_n(\epsilon), \quad \text{where} \quad
\epsilon = (1, (-1)^1, \dots, (-1)^{n-2} \, ; \, (-1)^n).
\]
\end{example}

We end with a natural question.

\begin{question}
Is the set $SSRPM_n(\epsilon)$ nonempty for all $n$ and sign patterns
$\epsilon \in \{ \pm 1 \}^n$?
(It is mentioned in~\cite{Signs} that explicit matrices can be found for
all $n \leq 5$ and all $\epsilon$.)
\end{question}

\end{document}